\documentclass[12pt]{amsart}
\usepackage{amsfonts,latexsym,rawfonts,amsmath,amssymb,amsthm,mathrsfs}
\usepackage[colorlinks=true,urlcolor=blue,linkcolor=blue,anchorcolor=blue,citecolor=blue]{hyperref}
\usepackage{graphicx}
\usepackage{comment}
\usepackage[pagewise]{lineno}
\usepackage{booktabs} %ŠÌ¡Â??¡Àšª????š®?šŠ????š²ššYŠÌ?????
\usepackage{multirow}
\usepackage{array}
\usepackage{bm}
\numberwithin{equation}{section}

\RequirePackage{color}
\theoremstyle{plain}

\newtheorem{theorem}{Theorem}[section]

\newtheorem{lemma}{Lemma}[section]
\newtheorem{problem}{Problem}[section]
\theoremstyle{Corollary}
\newtheorem{cor}[theorem]{Corollary}

\newtheorem{definition}[theorem]{Definition}

\newcommand{\beq}{\begin{equation}}
	\newcommand{\eeq}{\end{equation}}
\newcommand{\beqs}{\begin{eqnarray*}}
	\newcommand{\eeqs}{\end{eqnarray*}}
\newcommand{\beqn}{\begin{eqnarray}}
	\newcommand{\eeqn}{\end{eqnarray}}
\newcommand{\beqa}{\begin{array}}
	\newcommand{\eeqa}{\end{array}}

\def\S{\mathbb S}
\def\R{\mathbb R}

\def\K{\mathcal K}

\begin{document}
	
	\title{Flow by Gauss curvature to the $L_p$-Gaussian Minkowski problem}

	\author{Weimin Sheng}
	\address{Weimin Sheng: School of Mathematical Sciences, Zhejiang University, Hangzhou 310027, China.}
	\email{weimins@zju.edu.cn}
	
	\author{Ke Xue}
	\address{Ke Xue: School of Mathematical Sciences, Zhejiang University, Hangzhou 310027, China.}
	\email{21935078@zju.edu.cn}
	
	
	
	\thanks{This paper is a revised version of an earlier paper submitted on August 21, 2021.  In this new version we added Theorem 1.4 for the case $ -n-1<p\le 0$  and made some revisions. }

	\begin{abstract}
		In this paper
		we study the $L_p$-Gaussian Minkowski problem, which arises in the $L_p$-Brunn-Minkowski theory in Gaussian probability space.  We use Aleksandrov's variational method with Lagrange multipliers to prove the existence of the logarithmic Gauss Minkowski problem. 
		We construct a suitable Gauss curvature flow of closed, convex hypersurfaces in the Euclidean space $\R^{n+1}$, and prove its long-time existence and converges smoothly to a smooth solution of the normalized $L_p$ Gaussian Minkowski problem in cases  of $p>0$ and $-n-1<p\leq 0$ with even prescribed function respectively.
		We also provide a parabolic proof in the smooth category to the $L_p$-Gaussian Minkowski problem in cases of $p\geq n+1$ and $0<p<n+1$ with even prescribed function, respectively. 
		
	\end{abstract}
	\maketitle
	%%%%%%%%%%%%%%%%%%%%%%%%%%%%
	%%%%%%%%%%%%%%%%%%
	
	\baselineskip16pt
	\parskip3pt
	\section{Introduction}
	The Minkowski problem is a characterization problem for a geometric measure generated
	by convex bodies. The study of the Minkowski problem has a long history and its resolution has led  to a series of influential works, such as Minkowski\cite{1903Minkowski}, Lewy\cite{1938Lewy}, Nirenberg\cite{1953Nirenberg}, Pogorelov\cite{1978Pogorelov} and Cheng-Yau \cite{1976ChengYau}, etc.. The $L_p$ surface
	area measure and its associated Minkowski problem in the $L_p$ Brunn-Minkowski theory
	were introduced by Lutwak\cite{1993Lutwak} in 1993.  Since then, the $L_p$-Minkowski problem
	has been studied extensively in \cite{ 2005Hug, 2006ChWang, 2013BLYZ, 2015Huang, 2015JianWang,2015ZhuG, 2019BIS, SY20}, and so on.
	Recently, Huang-Xi-Zhao in \cite{2020HXZ} introduced a Minkowski problem in Gaussian probability space. In order to explain the problem, we introduce some notations as follows.
	
	We say $ K \subset \R^{n+1}$ is a convex body if it is a compact convex set with nonempty
	interior. Denote by $\K$ the set of all convex bodies in $\R^{n+1}$ containing the origin $O$, $\K_0$ the convex bodies that contain the origin in interior, and $\K_e$ the $O$-symmetry convex bodies.
	For $K\in \K$, we define its radial function $r_K: S^n\rightarrow [0, \infty)$ and support function $h_K: S^n\rightarrow [0, \infty)$, respectively, by
	\[
	r_K(x)=\max\{a\in \R: a\cdot x\in K\}\, \, \textrm{and}\, \, h_K(x)=\max\{x\cdot y: y\in K\}, \, \, x\in S^n,
	\]
	where $x\cdot y$ denotes the inner product in $\R^{n+1}$.
	
	For $K\in \K$, let $\partial K$ be its boundary.
	The Gauss map of $\partial K$, denoted by $\nu_{K}: \partial K\rightarrow S^n$ is defined as follows:
	for $y\in \partial K$,
	\[
	\nu_K(y)=\{x\in S^n: x\cdot y=h_K(x)\}.
	\]
	Let $\nu_K^{-1}: S^n\rightarrow \partial K$
	be the reverse Gauss map such that
	\[
	\nu_K^{-1}(x)=\{ y\in S^n: x\cdot y=h_K(x)\}, \, x\in S^n.
	\]
	
	%Denote by $\alpha_K: S^n\rightarrow S^n$ the radial Gauss image of $K$. That is
	%\[
	%\alpha_K(\xi)=\{x\in S^n: x\in \nu_K(r_K(\xi)\xi)\}, \, \xi\in S^n.
	%\]
	%Define $\alpha^*_K: S^n\rightarrow S^n$, the reverse radial Gauss image of $K$ as follows: for any Borel set $E\subset S^n$,
	%\[
	%\alpha^*_K(E)=\{\xi\in S^n: r_K(\xi)\xi\in \nu_K^{-1}(E)\}.
	% \]
	We often omit the subscript $K$ in $r_K$, $h_K$, $\nu_K$, $\nu_K^{-1}$,
	%$\alpha_K$ and $\alpha_K^*$
	if no confusion occurs.
	
	Now the Gaussian volume functional, denoted by $\gamma_{n+1}$, is defined by
	\[
	\gamma_{n+1}(K)=\frac{1}{(\sqrt{2\pi})^{n+1}}\int_{K}e^{\frac{-|x|^2}{2}}dx.
	\]
	Using the variational formula from \cite{2016HLYZ}, the variational formula of $\gamma_{n+1}$ is
	obtained as
	\[\lim_{t \to 0}\frac{\gamma_{n+1}(K+tL)-\gamma_{n+1}(K)}{t}=\int_{\mathbb{S}^{n}}h_L(v)dS_{\gamma_{n+1},K}(v),\]
	for any convex bodies $K, L\in \K$, where $K+t\cdot L=\{x+ty: x\in K, y\in L\}$,  and $S_{\gamma_{n+1},K}$ is the Gaussian surface area measure,
	\[
	S_{\gamma_{n+1},K}(\eta)=\frac{1}{(\sqrt{2\pi})^{n+1}}\int_{\nu_{K}^{-1}(\eta)}e^{\frac{-|x|^2}{2}}d\mathcal{H}^{n}(x), \, \, \eta\subset S^n.
	\]
	The Gaussian Minkowski problem asks that given a finite Borel measure $\mu$, what are
	the necessary and sufficient conditions on $\mu$ so that there exists a convex body $K$ with the origin
	$o \in int (K)$ such that
	\[
	\mu=S_{\gamma_{n+1},K}?
	\]
	If $K$ exists, to what extent is it unique?
	
	It is shown in \cite{2020HXZ} that for any finite even Borel measure $\mu$ on $S^{n}$ which is not concentrated
	on any closed hemisphere and $|\mu|<\frac{1}{\sqrt{2\pi}}$, there exists a unique origin-symmetry
	convex body with $\gamma_{n+1}(K) > \frac{1}{2}$ such that $S_{\gamma_{n+1},K}=\mu$.
	
	Gaussian Minkowski problem is different from the Minkowski
	problem in Lebesgue measure space in the following aspects. By the works of  Ball \cite{1993Ball}
	and Nazarov \cite{2003Nazarov}, the allowable $\mu$ in the Gaussian Minkowski problem cannot have an arbitrarily big total mass. In fact the Gaussian surface
	area of any convex set in $\mathbb{R}^{n+1}$ is not more than $4(n+1)^{\frac{1}{4}}$.
	Furthermore, by the definition of Gaussian surface area, we can know that not only
	small convex body, but also large one, can has smaller Gaussian surface areas. For instance, let $B_r$ be a centered ball with radius $r$, the Gaussian area density of $B_r$ is $f_r=\frac{1}{(\sqrt{2\pi})^{n+1}}e^{\frac{-r^2}{2}}r^n$. One can observe that $\frac{1}{(\sqrt{2\pi})^{n+1}}e^{\frac{-r^2}{2}}r^n\rightarrow0$ as $r\rightarrow0$ or $r\rightarrow\infty$. It is easy to see $f_r<\frac{1}{(\sqrt{2\pi})^{n+1}}$ when $n<3$. Secondly, Gaussian probability measure has neither
	translation invariance nor homogeneity. 
	
	It is natural to introduce the $L_p$-Brunn-Minkowski theory in Gaussian probability
	space, see \cite{2021Liu}. The $L_p$-Gaussian surface area measures can be defined as 
	\[S_{p,\gamma_{n+1},K}(\eta)=\frac{1}{(\sqrt{2\pi})^{n+1}}\int_{\nu_{K}^{-1}(\eta)}e^{\frac{-|x|^2}{2}}(x\cdot\nu_{K}(x))^{1-p}d\mathcal{H}^{n}(x).\]
	The $L_p$-Gaussian Minkowski problem can be stated as follows:{\it{
	Given a finite Borel measure $\mu$ on
	$S^n$, what are the necessary and sufficient conditions on $\mu$ does there exist a convex
	body $K$ in $\mathbb{R}^{n+1}$ such that $L_p$-Gaussian surface area measure $S_{p,\gamma_{n+1},K}(\cdot)=\mu(\cdot)$? If $K$
	exists, is it unique?}}

	If $f$ is the density function of the given measure $\mu$, the $L_p$-Gaussian Minkowski problem becomes the following Monge-Amp\`{e}re equation on $\mathbb{S}^n $:
	\begin{equation}\label{flow01}
		\frac{1}{(\sqrt{2\pi})^{n+1}}e^{-{\frac{|\nabla h|^2+h^2}{2}}}h^{1-p}\det{(h_{ij}+h\delta_{ij})}=f,
	\end{equation}
	where $h$, $f$ are the support function and the prescribed function, respectively.
	
	In \cite{2021Liu}, Liu has proved the following result.\\
	%\begin{theorem}
	{\bf{Theorem A}} {\it
		Let $\alpha\in(0,1)$, for $p\geq1$, $f\in C^{2,\alpha}(\mathbb{S}^n)$ is a positive even function and satisfies $|f|_{L_1}<\sqrt{\frac{2}{\pi}}r^{-p}ae^{\frac{-a^2}{2}}$, where $r$ and $a$ are chosen such that $\gamma_{n+1}(rB)=\gamma_{n+1}(P)=\frac{1}{2}$, symmetry trip $P=\{x\in\mathbb{R}^{n+1}:|x_1|\leq a\}$. Then
		there exists a unique $C^{4,\alpha}$ convex body $K\in\mathcal{K}^{n+1}_e$ with $\gamma_{n+1}(K)>\frac{1}{2}$, its support function $h_K$ satisfies the equation (\ref{flow01}).}
	%\end{theorem}

	In this paper, we will improve the result in Theorem \ref{mainthm}.
	% we only need $f$ is a smooth positive function on $\mathbb{S}^n$, we can obtain the existence and uniqueness of smooth solution to $L_p$-Gaussian Minkowski problem in case of $p\geq n+1$. If $0<p<n+1$, $f$ is even function satisfies (\ref{jia1.12}) and origin-symmetric initial condition, in Theorem (\ref{mainthm2}) we can obtain the existence result, but in this case the uniqueness of smooth solution can't get.
	
	Since the $L_p$ Gaussian surface area measures has neither translation invariance nor homogeneity, we may also consider the normalized $L_p$-Gaussian Minkowski problem in the following way:
	\begin{problem}
		Given a finite Borel measure $\mu$ on
		$S^n$, what are the necessary and sufficient conditions on $\mu$ does there exist a convex
		body K in $\mathbb{R}^{n+1}$ and a positive constant $c$ such that $\mu(\cdot)=cS_{p,\gamma_{n+1},K}(\cdot)$?
	\end{problem}
	Obviously, this problem becomes the following Monge-Amp\`{e}re equation on $\mathbb{S}^n $:
	\begin{equation}\label{yibanfangcheng1}
		ce^{-{\frac{|\nabla h|^2+h^2}{2}}}h^{1-p}\det{(h_{ij}+h\delta_{ij})}=f,
	\end{equation}
	where $f$ is the density of the measure $\mu$.

	In \cite{2021Liu}, by use of Aleksandrov's variational method, Liu has proved \\
	{\bf{Theorem B}} {\it For $p>0$, let $\mu$ be a nonzero finite Borel measure on $S^{n}$ and be not concentrated in any closed hemisphere. Then there exist a convex body $K\in \K_0$ and a positive constant $\lambda$ such that
		\[
		\mu=\frac{\lambda}{p}S_{p, \gamma_{n+1}, K}.
		\]
	}
	
	In this paper, we study the case $p=0$, i. e. the logarithmic Gauss Minkowski problem. We will use Aleksandrov's variational method to prove the existence of solution with Lagrange
	multipliers.
	\begin{definition} A finite Borel measure $\mu$ on $\mathbb{S}^n$ is said to satisfy \textbf{the subspace concentration
			inequality} if, for every subspace $\xi$ of $\mathbb{R}^{n+1}$, such that $0<\dim\xi<n+1$,
		\begin{equation}\label{def1}
			\mu(\xi\cap\mathbb{S}^n)\leq \frac{1}{n+1}\mu(\mathbb{S}^n)\dim\xi.
		\end{equation}
		%\indent The measure is said to satisfy \textbf{the subspace concentration condition} if in addition to satisfying
		%the subspace concentration inequality (\ref{def1}), whenever
		%\[ \mu(\xi\cap\mathbb{S}^n)= \frac{1}{n+1}\mu(\mathbb{S}^n)\dim\xi,\]
		%for some subspace $\xi$, then there exists a subspace $\xi^{'}$, that is complementary to $\xi$ in $\mathbb{R}^{n+1}$, so that also
		%\[ \mu(\xi^{'}\cap\mathbb{S}^n)= \frac{1}{n+1}\mu(\mathbb{S}^n)\dim\xi^{'},\]
		%or equivalently so that $\mu$ is concentrated on $\mathbb{S}^n\cap(\xi\cup\xi^{'})$.\\
		The measure $\mu$ on $\mathbb{S}^n$ is said to satisfy \textbf{the strict subspace concentration inequality} if
		the inequality in (\ref{def1}) is strict for each subspace $\xi\subset\mathbb{R}^{n+1}$, and $0< \dim\xi< n+1$.
	\end{definition}
	\begin{theorem}\label{p0}
		Let $p=0$, if $\mu$ is even finite Borel measure on $\mathbb{S}^n$ that satisfies the strict subspace concentration inequality. Then there exist an $O$-symmetric convex body K and a positive constant $c$ such that
		\[
		\mu=c S_{p,\gamma_{n+1},K}.
		\]
		
	\end{theorem}
	We also use the following normalised Gauss curvature flow to solve the existence of smooth solution for (\ref{yibanfangcheng1}). Let $M_0$ be a smooth, closed and convex hypersurface in $\mathbb{R}^{n+1}$ which encloses the origin, $n\geq 1$. We study following curvature flow
	\begin{equation}\left\{
		\begin{array}{l}
			\displaystyle \frac{\partial X}{\partial t}(x,t)=-\theta(t)e^{\frac{r^2}{2}}K(x,t){\langle X,\nu\rangle}^pf(\nu)\nu+X(x,t),\\
			X(x,0)=X_0(x),
		\end{array}\right.\label{yibanfangcheng2}
	\end{equation}
	where $K(\cdot,t)$ is the Gauss curvature of hypersurface $M_t$, parametrized by $X(\cdot,t):\mathbb{S}^{n}\rightarrow\mathbb{R}^{n+1}$,   $\nu(\cdot,t)$ is the unit outer normal at $X(\cdot,t)$, and $f$ is a given positive smooth function on $\mathbb{S}^{n}$,  $r=|X(x,t)|$ the distance from the point $X(x,t)$ to the origin, and $\theta(t)$ is defined as
	\[
	\theta(t)=\int_{\mathbb{S}^n}e^{-\frac{r^2}{2}}r^{n+1}(\xi,t)d\xi\bigg/\int_{\mathbb{S}^n}h^p(x,t)f(x)dx.
	\]
	$r$ can be regarded as a function of $\xi=\xi(x,t):=X(x,t)/|X(x,t)|\in\mathbb{S}^{n}$, which we call it as the radial function. 
	The following functional plays an important role in our argument
	\begin{equation}
		\Psi(M_t)=\left\{
		\begin{array}{lr}
			\displaystyle \frac{1}{p}\int_{\mathbb{S}^n}f(x)h^p(x,t)dx, \ \text{if}\ p\neq0.\\
			\displaystyle\int_{\mathbb{S}^n}f(x)\log h(x,t)dx,\ \text{if}\ p=0.
		\end{array}\right.\label{c}
	\end{equation}
	%\begin{equation}\label{c}
	%	\Psi(M_t)=\frac{1}{p}\int_{\mathbb{S}^n}f(x)h^p(x,t)dx, \ \text{if}\ p\neq0,
	%\end{equation}
	where $h(\cdot,t)$ is the support function of $M_t$. 
	We have the following theorem.
	
	\begin{theorem}\label{yibanfangchengthm}
		Let $M_0$ be a smooth, closed, uniformly convex hypersurface  in $\mathbb{R}^{n+1}$ enclosing the origin, and $\gamma_{n+1}(\Omega_0)\geq\frac{1}{2}$, where $\Omega_0$ denotes the convex body enclosed by $M_0$. Let $f$ be a smooth positive function on $\mathbb{S}^n$ and $p > 0$. Then the normalised flow (\ref{yibanfangcheng2}) has a
		unique smooth solution $M_t = X(\mathbb{S}^n, t)$ for any time, and a subsequence of $M_t$ converge smoothly to the smooth solution of (\ref{yibanfangcheng1}) with $\frac{1}{c}=\lim_{t_i\rightarrow\infty}\theta(t_i)>0$.
	\end{theorem}
	
	When $-n-1<p\leq 0$, we can aslo obtain the following theorem.
	\begin{theorem}\label{yibanfangchengthm2}
		\indent Let $M_0$ be as in Theorem \ref{yibanfangchengthm} and $M_0$ is origin-symmetric. Let $f$ be a smooth positive  even function on $\mathbb{S}^n$ and $-n-1<p\leq 0$. Then the normalised flow (\ref{yibanfangcheng2}) has a
		unique smooth solution $M_t = X(\mathbb{S}^n, t)$ for any time, and a subsequence of $M_t$ converge smoothly to the smooth solution of (\ref{yibanfangcheng1}) with $\frac{1}{c}=\lim_{t_i\rightarrow\infty}\theta(t_i)>0$.
	\end{theorem}

	%	In \cite{2020CLL}, a more wild geometric flow was considered, and a convergence result was obtained under the assumption that the initial convex body is origin-symmetric. In Theorem \ref{yibanfangchengthm}, the origin-symmetric condition is not required.

	Next, we use the Gauss curvature flow \eqref{yibanfangcheng2} to study the $L_p$-Gaussian Minkowski problem \eqref{flow01}. We take another value $\theta(t)= (\sqrt{2\pi})^{n+1}$, then the flow  (\ref{yibanfangcheng2}) becomes
	\begin{equation}\label{flow02}
		\left\{
		\begin{array}{l}
			\displaystyle \frac{\partial X}{\partial t}(x,t)=-(\sqrt{2\pi})^{n+1}e^{\frac{r^2}{2}}K(x,t){\langle X,\nu\rangle}^pf(\nu)\nu+X(x,t)\\
			X(x,0)=X_0(x)
		\end{array}\right.
	\end{equation}
	We consider the following functional in our argument.
	\begin{equation}\label{flow03}
		\Phi(M_t)=\frac{1}{p}\int_{\mathbb{S}^n}f(x)h^p(x,t)dx-\frac{1}{(\sqrt{2\pi})^{n+1}}\int_{\mathbb{S}^n}d\xi\int_0^{r(\xi,t)}e^{-\frac{s^2}{2}}s^nds, \ \text{if}\ p\neq0,
	\end{equation}
	%\begin{equation}\left\{
	%\begin{array}{l}
	%\displaystyle\Phi(M_t)=\frac{1}{p}\int_{\mathbb{S}^n}f(x)u(x,t)^pdx-\frac{1}{(\sqrt{2\pi})^{n+1}}\int_{\mathbb{S}^n}d\xi\int_0^{r(\xi,t)}e^{-\frac{s^2}{2}}s^nds, \text{if} p\neq0 \\
	%\displaystyle\Phi(M_t)=\int_{\mathbb{S}^n}f(x)\log u(x,t)dx-\frac{1}{(\sqrt{2\pi})^{n+1}}\int_{\mathbb{S}^n}d\xi\int_0^{r(\xi,t)}e^{-\frac{s^2}{2}}s^nds,\text{if} p=0
	%\end{array}\right.\label{flow03}
	%\end{equation}
	where $h(\cdot,t)$ and $r(\cdot,t)$ are  the support function and radial function of $M_t$, respectively. 
	By studying this functional and the a priori estimates of the solutions to the flow (\ref{flow02}), we obtain the following convergence result for the asymptotic flow.
	
	\begin{theorem}\label{mainthm}
		\indent Let $M_0$ be a smooth, closed, uniformly convex hypersurface  in $\mathbb{R}^{n+1}$ enclosing the origin. Let $f$ is a smooth positive function on $\mathbb{S}^n$. If $p> n+1$ or $p=n+1$ with $f<\frac{1}{(\sqrt{2\pi})^{n+1}}$, then flow (\ref{flow02}) has  a unique smooth, uniformly convex solution $M_t$ for all time $t>0$. When $t\rightarrow\infty$, a subsequence of $M_t$ converge smoothly to the unique smooth solution of (\ref{flow01}), which is the minimiser of the functional (\ref{flow03}).
	\end{theorem}
		\begin{cor}\label{Conver}
		\indent Let $M_0$ be a smooth, closed, uniformly convex hypersurface  in $\mathbb{R}^{n+1}$, enclosing the origin. If $f=\frac{1}{2(\sqrt{2\pi})^{n+1}}$ and $p\geq n+1$, then the hypersurface $M_t$ converge exponentially to a sphere centered at the origin in the $C^\infty$ topology.
	\end{cor}
	\indent When $0<p<n+1$, in order that the flow (\ref{flow02}) converges to a solution of (\ref{flow01}), we assume that $f\in C^{\infty}(\mathbb{S}^n;\mathbb{R}_{+})$ and satisfies 
	%\begin{equation}\label{jia1.10}
	% \int_{\mathbb{S}^n}f=o_n:=|\mathbb{S}^n|,
	%\end{equation}
	%\begin{equation}\label{jia1.11}
	% \int_{\omega}f<|\mathbb{S}^n|-|\omega^{*}|,
	%\end{equation}
	\begin{equation}\label{jia1.12}
		\gamma_{n+1}(\Omega_0)>\frac{1}{p} \int_{\mathbb{S}^n}f(x)h^p(x,0)dx,
	\end{equation}
	%for any spherically convex subset $\omega\subset\mathbb{S}^n$. Here $|\cdot|$ denotes the $n-$dimensional Hausdorff measure, and $\omega^{*}\subset\mathbb{S}^n$ is the dual set of $\omega$, namely $\omega^{*}=\{\xi\in\mathbb{S}^n:\quad x\cdot\xi\leq0, \forall x\in\omega\}$, and%
	where $\Omega_0$ is the convex body enclosed by $M_0$.
	
	\begin{theorem}\label{mainthm2}
		\indent Let $M_0$ be as in Theorem \ref{mainthm} and $M_0$ is origin-symmetric. Let $0<p<n+1$. If $f$ is even function and 
		\eqref{jia1.12} hold, then flow (\ref{flow02}) has  a unique smooth, uniformly convex solution $M_t$ for all time $t>0$. When $t\rightarrow\infty$, a subsequence of $M_t$ converge smoothly to the smooth even solution of (\ref{flow01}), which is the minimiser of the functional (\ref{flow03}).
	\end{theorem}
	
	By Theorem \ref{mainthm} and Theorem \ref{mainthm2}, we have the following existence results for equation (\ref{flow01}).
	\begin{theorem}\label{mainthm3}
		Let $f$ be a smooth and positive function on the sphere $\mathbb{S}^n$.
		\begin{itemize}
			\item[(i)] If $p> n+1$, there is a unique smooth, uniformly convex solution to (\ref{flow01}).
			\item[(ii)] If $p= n+1$ and $f<\frac{1}{(\sqrt{2\pi})^{n+1}}$, there is a unique smooth, uniformly convex solution to (\ref{flow01}).
			\item[(iii)] If $0<p<n+1$, $f$ is even function and satisfies \eqref{jia1.12}, there is an origin-symmetric solution to \eqref{flow01}.
		\end{itemize}
	\end{theorem}

	This paper is organised as follows.
	In Section 2, we collect some properties of convex body and convex
	hypersurfaces, and show that the flow (\ref{flow02}) can be reduced to a scalar parabolic equation
	of Monge-Amp\`{e}re type, via the support function or the radial function. We will also
	show in Section 2 that the functionals \eqref{c} and \eqref{flow03} are non-increasing along the flows \eqref{yibanfangcheng2} and  \eqref{flow02}, respectively.
	In Section 3, Aleksandrov's variational method is applied to obtain the origin-symmetry solution with Lagrange multipliers for case $p=0$.
	In Section 4, we give the proof of Theorem \ref{yibanfangchengthm} and \ref{yibanfangchengthm2}.
	The proofs of Corollary \ref{Conver}, Theorems \ref{mainthm} and \ref{mainthm2}
	will be presented in Section 5.

	\section{Preliminaries}\label{sec2}
	In this section, we continue to give a brief review of some relevant notions about convex bodies and recall some basic properties of convex hypersurfaces.
	Let 
	%$\mathbb{R}^{n+1}$ be the $(n+1)$-dimensional Euclidean space, $\mathbb{S}^n$ be the unit sphere in $\mathbb{R}^{n+1}$, 
	$C^+(\S^{n})$ and $C^+_e(\mathbb{S}^n)$ be the sets of the positive and  positive even functions defined on $\mathbb{S}^n$, repectively.
	%Convex body in $\mathbb{R}^{n+1}$ is compact convex set with nonempty interior denoted by $\mathcal{K}^{n+1}$.Let $\partial K$ denote the boundary of K, $\mathcal{K}_o^{n+1}$  be convex bodies that contain the origin in interior, and $\mathcal{K}_e^{n+1}$ denote the o-symmetry convex bodies.\\
	% The support function $h_Q$ of a compact convex subset $Q$ in $\mathbb{R}^{n+1}$ is defined as
	%\[h_Q(y)=\max_{x\in Q}y\cdot x,\] for $y\in\mathbb{R}^{n+1}$, where $\cdot$ represents the inner product.\\
	%\indent For $Q\in \mathcal{K}_o^{n+1}$, the radial function $r_Q$ is defined by
	%\[r_Q(x)=\max \{\lambda:\lambda x\in Q\},\]
	%for $x\in \mathbb{R}^{n+1}\backslash \{0\}$. A simple observation allows us to obtain
	%\[\partial Q=\{r_Q(u)u:u\in\mathbb{S}^n\}.\]
	Let $h\in C^+(\mathbb{S}^n)$, the Wulff shape $[h]$ generated by $h$ is a convex body defined by
	\[
	[h]=\{x\in\mathbb{R}^{n+1}:x\cdot v\leq h(v),\forall v\in \mathbb{S}^n\}.
	\]
	Obviously, if $K\in \mathcal{K}_o$, $[h_K]=K$.
	
	For a compact convex subset $K\in \K$  and $v\in\mathbb{S}^n$, the supporting hyperplane $H(K,v)$ of $K$ at $v$ is given by
	\[
	H(K,v)=\{x\in K: x\cdot v=h_K(v)\}.
	\]
	The boundary point of $K$ which only has one supporting hyperplane called regularity
	point, otherwise, it is a singular point. The set of singular points is denoted as $\sigma_K$, it
	is well known that $\sigma_K$ has spherical Lebesgue measure $0$.
	
	For $x\in \partial K \backslash \sigma_K$, its Gauss map $\nu_K$ is represented by
	\[\nu_K(x)=\{v\in \mathbb{S}^n: x\cdot v=h_K(v)\}.\]
	Corresponding, for a Borel set $\eta\subset\mathbb{S}^n$, its reverse Gauss map is denoted by $\nu_K^{-1}$,
	\[
	\nu_K^{-1}(\eta)=\{x\in \partial K: \nu_K(x)\in\eta\}.
	\]
	For the Borel set $\eta\subset\mathbb{S}^n$, its surface area measure is defined as
	\[
	S_K(\eta)=\mathcal{H}^n(\nu_K^{-1}(\eta)),
	\]
	where $\mathcal{H}^n$ is $n$-dimensional Hausdorff measure.
	
	For a Borel set $\omega\subset\mathbb{S}^n$, $\boldsymbol{\alpha}_K(\omega)$ denotes its radial Gauss image and is defined as
	\[
	\boldsymbol{\alpha}_K(\omega)=\{v\in \mathbb{S}^n: r_K(u)(u\cdot v)=h_K(v)\quad \text{for some}\ u\in\omega\}, 
	\]
	when the Borel set $\omega$ has only one element $u$, we will abbreviate $\boldsymbol{\alpha}_K(\{u\})$ as $\boldsymbol{\alpha}_K(u)$. The subset of $\mathbb{S}^n$ which make $\boldsymbol{\alpha}_K(u)$ contain more than one element denoted by $\omega_K$ for each $u\in\omega$. The set $\omega_K$ has spherical Lebesgue measure $0$.
	
	The radial Guass map of $K$ is a map denoted by $\alpha_K(u)$, the only difference between $\alpha_K(u)$ and $\boldsymbol{\alpha}_K(u)$ is that the former is defined on $\mathbb{S}^n\backslash \omega_K$ not on $\mathbb{S}^n$ which lead
	to $\boldsymbol{\alpha}_K(u)$ may have many elements but $\alpha_K(u)$ has only one. In other words, if $\boldsymbol{\alpha}_K(u)=\{v\}$, then $\boldsymbol{\alpha}_K(u)=\alpha_K(u)$.
	
	Let $M$ be a smooth, closed, uniformly convex hypersurface in $\mathbb{R}^{n+1}$.
	Assume that $M$ is parametrized by the inverse Gauss map $X: \mathbb{S}^n \rightarrow M$. The support function $h:\mathbb{S}^n \rightarrow \mathbb{R}$ of $M$ is defined by
	\begin{equation}\label{1}
		h(x)=\sup\left\{{\langle x,y\rangle}:y\in M\right\}.
	\end{equation}
	\indent The supremum is attained at a point $y$ such that $x$ is the outer normal of $M$ at $y$. it is easy to check that
	\begin{equation}\label{2}
		y=h(x)x+\nabla h(x),
	\end{equation}
	where $\nabla$ is the covariant derivative with respect to the standard matric $e_{ij}$ of the sphere $\mathbb{S}^n$. Here
	\begin{equation}\label{3}
		r=|y|=\sqrt{h^2+|\nabla h|^2}.
	\end{equation}
	\indent The second fundamental form of $M$ is given that
	\begin{equation}\label{4}
		b_{ij}=h_{ij}+he_{ij},
	\end{equation}
	where $h_{ij}=\nabla_{ij}^2h$ denotes the second order covariant derivative of $h$ with respect to the spherical metric $e_{ij}$. By Weingarten's formula,
	\begin{equation}\label{5}
		e_{ij}={\langle \frac{\partial \nu}{\partial x^i}, \frac{\partial \nu}{\partial x^j}\rangle}=b_{ik}g^{kl}b_{jl},
	\end{equation}
	where $g_{ij}$ is the metric of $M$ and $g^{ij}$ its inverse. It follows from (\ref{4}) and (\ref{5}) that the principal radii of curvature of $M$, under a smooth local orthonormal frame on $\mathbb{S}^n$, are the eigenvalues of the matrix
	\begin{equation}\label{6}
		b_{ij}=h_{ij}+h\delta_{ij}.
	\end{equation}
	The Gauss curvature is given by
	\begin{equation}\label{7}
		K=\frac{1}{\det(h_{ij}+h\delta_{ij})}=S_n^{-1}(h_{ij}+h\delta_{ij}),
	\end{equation}
	where
	\[
	S_k=\sum_{1\le i_1<\cdot\cdot\cdot<i_k\le n}\lambda_{i_1}\cdot\cdot\cdot\lambda_{i_k},
	\]
	denotes the $k$-th elementary symmetric polynomial of $\lambda=(\lambda_1, ..., \lambda_n)$.
	
	Let $X(\cdot,t)$ be a smooth solution to the flow (\ref{yibanfangcheng2}) and $h(\cdot,t)$ its support function, then the flow \eqref{yibanfangcheng2} can be reduced to the initial  value problem of the support function $h$:
	\begin{equation}
		\left\{
		\begin{array}{l}
			\displaystyle \frac{\partial h}{\partial t}(x,t)=-\theta(t) e^{\frac{r^2}{2}}K(x,t)h^pf+h(x,t),\\
			h(x,0)=h_0(x),
		\end{array}\right.\label{8}
	\end{equation}
	where $r=\sqrt{h^2+|\nabla h|^2}(x,t)$, $h_0$ is the support function of the initial hypersurface $M_0$.
	
	Since $M$ encloses the origin, it can be parametrized via the radial function $r:\mathbb{S}^n\rightarrow\mathbb{R}_{+}$,
	\[
	M=\left\{r(\xi)\xi:\xi\in\mathbb{S}^n \right\}.
	\]
	Then the following formulea are well-known:
	\begin{equation}\label{9}
		\nu=\frac{r\xi-\nabla r}{\sqrt{r^2+|\nabla r|^2}},
	\end{equation}
	and 
	\begin{equation}\label{10}
		\begin{array}{l}
			\displaystyle g_{ij}=r^2e_{ij}+r_ir_j,\\
			\displaystyle b_{ij}=\frac{r^2e_{ij}+2r_ir_j-rr_{ij}}{\sqrt{r^2+|\nabla r|^2}}.
		\end{array}
	\end{equation}
	Set
	\begin{equation}\label{11}
		v=\frac{r}{h}=\sqrt{1+|\nabla \log r|^2},
	\end{equation}
	where the last equality follows by multiplying $\xi$ to both sides of (\ref{9}).
	Since
	\begin{equation}\label{d}
		\frac{1}{r(\xi,t)}\frac{\partial r(\xi,t)}{\partial t}=\frac{1}{h(x,t)}\frac{\partial h(x,t)}{\partial t},
	\end{equation}
	%\[\frac{1}{r(\xi,t)}\frac{\partial r(\xi,t)}{\partial t}=\frac{1}{h(x,t)}\frac{\partial h(x,t)}{\partial t},\]
	(see \cite{2020LSW} for the proof),  the flow (\ref{flow02}) can be also described by the following scalar equation of $r(\cdot,t)$,
	\begin{equation}\left\{
		\begin{array}{l}
			\displaystyle \frac{\partial r}{\partial t}(\xi,t)=\frac{-\theta(t) e^{\frac{r^2}{2}}K(\xi,t)r^pf}{(1+|\nabla \log r|^2)^{\frac{p-1}{2}}}+r(\xi,t)\\
			r(\cdot,0)=r_0.
		\end{array}\right.\label{12}
	\end{equation}
	where $r_0$ is the radial function of $M_0$, $K(\xi,t)$ denotes the Gauss curvature at $r(\xi,t)\xi\in M_t$ and $f$ takes its value at $\nu=\nu(\xi,t)$ which is given by (\ref{9}). By (\ref{10}) we have, under a local orthonormal frame on $\mathbb{S}^n$,
	\begin{equation}\label{13}
		K=\frac{\det b_{ij}}{\det g_{ij}}=v^{-n-2}r^{-3n}\det(r^2\delta_{ij}+2r_ir_j-rr_{ij}).
	\end{equation}
	
	When we take another value $\theta(t)=(\sqrt{2\pi})^{n+1}$, then equation \eqref{yibanfangcheng2} becomes \eqref{flow02},  the corresponding equations \eqref{8} and \eqref{12} turn out to be 
	\begin{equation}\label{8-1}
		\left\{
		\begin{array}{l}
			\displaystyle \frac{\partial h}{\partial t}(x,t)=-(\sqrt{2\pi})^{n+1}e^{\frac{r^2}{2}}K(x,t)h^pf+h(x,t),\\
			h(x,0)=h_0(x),
		\end{array}\right.
	\end{equation}
	and 
	\begin{equation}\left\{
		\begin{array}{l}
			\displaystyle \frac{\partial r}{\partial t}(\xi,t)=\frac{-(\sqrt{2\pi})^{n+1}e^{\frac{r^2}{2}}K(\xi,t)r^pf}{(1+|\nabla \log r|^2)^{\frac{p-1}{2}}}+r(\xi,t)\\
			r(\cdot,0)=r_0.
		\end{array}\right.\label{12-1}
	\end{equation}
	respectively.  It is clear that both of the equations \eqref{8} and \eqref{8-1} are parabolic Monge-Amp\'ere type, their solutions exist for a short time. Therefore the flows \eqref{yibanfangcheng2} and \eqref{flow02} have short time solutions.

	Given any $\omega\subset\mathbb{S}^n$, let $\mathcal{C}=\mathcal{C}_{M,\omega}$ be the "cone-like" region with the vertex at the origin and the base $\nu^{-1}(\omega)\subset M$, namely
	\[\mathcal{C}:=\{z\in\mathbb{R}^{n+1}:\quad z=\lambda\nu^{-1}(x), \lambda\in[0,1], x\in\omega\}.\]

	It is well-known that the volume element of $\mathcal{C}$ can be expressed by
	\begin{equation}\label{140}
		d\text{Vol}(\mathcal{C})=\frac{1}{n+1}\frac{h(x)}{K(p)}dx=\frac{1}{n+1}r^{n+1}(\xi)d\xi,
	\end{equation}
	where $p=\nu^{-1}(x)\in M$ and $\xi, x$ are associated by
	\begin{equation}\label{150}
		r(\xi)\xi=h(x)x+\nabla h(x),
	\end{equation}
	namely $p=\nu^{-1}(x)=r(\xi)\xi$. By the second equality in (\ref{140}), we find that the determinant of the Jacobian of the mapping $x\mapsto\xi$ is given by
	\begin{equation}\label{Jac}
	\Big|\frac{d\xi}{dx}\Big|=\frac{h(x)}{r^{n+1}(\xi)K(p)}
	\end{equation}

%	\[\Big|\frac{d\xi}{dx}\Big|=\frac{h(x)}{r^{n+1}(\xi)K(p)}.\]

	We can show the Gaussian volume unchanged along the flow (\ref{yibanfangcheng2}). In fact, we have the following lemma.
	\begin{lemma}\label{gamma-invariant}
		Let $X(\cdot, t)$ be a smooth solution to the flow (\ref{yibanfangcheng2}) with $t\in[0,T)$, and for
		each $t > 0, M_t = X(\mathbb{S}^n, t)$ be a smooth, closed and uniformly convex hypersurface. Suppose that the origin lies in the interior of the convex body $\Omega_t$ enclosed by $M_t$ for all $t\in[0,T)$. Then
		
		\[\gamma_{n+1}(\Omega_t)=\gamma_{n+1}(\Omega_0),\ \forall t \in [0,T).\]
	\end{lemma}
	\begin{proof}
		Let $r(\cdot, t)$ and $h(\cdot,t)$ be the radial and support function of $\Omega_t$. Applying polar coordinates, \eqref{d}, \eqref{8} and \eqref{Jac}, 
		we have
		\begin{align*}
			\frac{d}{dt}\gamma_{n+1}(\Omega_t)&=\frac{1}{(\sqrt{2\pi})^{n+1}} \frac{d}{dt}\int_{\mathbb{S}^n}\int_0^{r(\xi, t)}e^{-\frac{r^2}{2}}r^{n}dr d\xi\nonumber\\
			&=\frac{1}{(\sqrt{2\pi})^{n+1}} \int_{\mathbb{S}^n} e^{-\frac{r^2}{2}}r^{n}r_td\xi\nonumber\\
			&=\frac{1}{(\sqrt{2\pi})^{n+1}}\int_{\mathbb{S}^n}e^{-\frac{r^2}{2}}r^{n+1}\frac{h_t}{h}d\xi\nonumber\\
			&=\frac{1}{(\sqrt{2\pi})^{n+1}}\int_{\mathbb{S}^n}e^{-{\frac{|\nabla h|^2+h^2}{2}}}\frac{h_t}{K}dx\nonumber\\
			&=-\frac{1}{(\sqrt{2\pi})^{n+1}}\theta(t)\int_{\mathbb{S}^n}fh^pdx+\frac{1}{(\sqrt{2\pi})^{n+1}}\int_{\mathbb{S}^n}e^{-{\frac{|\nabla h|^2+h^2}{2}}}\frac{h}{K}dx\nonumber\\
			&= 0.
		\end{align*}
		The last equality holds from the definition of $\theta(t)$. This finishes the proof.
	\end{proof}
	
	The next two lemmata show that the functionals (\ref{c}) and (\ref{flow03}) are non-increasing along the flows (\ref{yibanfangcheng2}) and (\ref{flow02}), respectively.
	
	\begin{lemma}\label{2.2}
		The functional (\ref{c}) is non-increasing along the flow (\ref{yibanfangcheng2}). Namely $\frac{d}{dt}\Psi(M_t)\leq 0$, and the equality holds if and only if $M_t$ satisfies the elliptic equation (\ref{yibanfangcheng1}).
	\end{lemma}
	\begin{proof}
		By (\ref{8}), we have
		\begin{align*}
			\frac{d}{dt}\Psi(M_t)&= \int_{\mathbb{S}^n}f(x)h^{p-1}h_tdx\nonumber\\
			&=-\theta(t)\int_{\mathbb{S}^n}e^{{\frac{|\nabla h|^2+h^2}{2}}}K{f}^2h^{2p-1}dx+\int_{\mathbb{S}^n}f(x)h^{p}dx\nonumber\\
			&=\{\int_{\mathbb{S}^n}f(x)h^{p}dx\}^{-1}\{-\int_{\mathbb{S}^n}e^{-{\frac{|\nabla h|^2+h^2}{2}}}\frac{h}{K}dx\int_{\mathbb{S}^n}e^{\frac{|\nabla h|^2+h^2}{2}}Kf^2h^{2p-1}dx+(\int_{\mathbb{S}^n}f(x)h^{p}dx)^2\}\nonumber\\
			&\leq 0,
		\end{align*}
		where the last inequality holds from H\"{o}lder inequality, and the equality holds if and only if
		\[
		fh^{p-1}=c\frac{1}{K}e^{-{\frac{|\nabla h|^2+h^2}{2}}}.
		\]
		Namely $M_t$ satisfies (\ref{yibanfangcheng1}) with $\frac{1}{c}=\theta(t)$.
	\end{proof}

	\begin{lemma}\label{2.1}%Lemma 2.1
		The functional (\ref{flow03}) is non-increasing along the flow (\ref{flow02}). Namely $\frac{d}{dt}\Phi(M_t)\leq 0$, and the equality holds if and only if $M_t$ satisfies the elliptic equation (\ref{flow01}).
	\end{lemma}
	
	\begin{proof}
		By use of (\ref{flow03}) and \eqref{8-1}, we have  
		\begin{align*}
			\frac{d}{dt}\Phi(M_t)&= \int_{\mathbb{S}^n}f(x)h^{p-1}h_tdx-\frac{1}{(\sqrt{2\pi})^{n+1}}\int_{\mathbb{S}^n}e^{-\frac{r^2}{2}}r^n(\xi,t)r_t(\xi,t)d\xi\nonumber\\
			&=\int_{\mathbb{S}^n}f(x)h^{p-1}h_t-\frac{1}{(\sqrt{2\pi})^{n+1}}e^{-{\frac{|\nabla h|^2+h^2}{2}}}\frac{h_t}{K}dx\nonumber\\
			&=-(\sqrt{2\pi})^{n+1}\int_{\mathbb{S}^n}e^{\frac{r^2}{2}}Kh(fh^{p-1}-\frac{1}{(\sqrt{2\pi})^{n+1}}\frac{1}{K}e^{-{\frac{|\nabla h|^2+h^2}{2}}})^2dx\nonumber\\
			&\leq 0.
		\end{align*}
		Clearly $\frac{d}{dt}\Phi(M_t)=0$ if and only if
		\[fh^{p-1}=\frac{1}{(\sqrt{2\pi})^{n+1}}\frac{1}{K}e^{-{\frac{|\nabla h|^2+h^2}{2}}}.\]
		Namely $M_t$ satisfies (\ref{flow01}).
	\end{proof}

	The next lemma is necessary to prove the $L_p$-variational formula for $p=0$.
	\begin{lemma}\label{p1}
		For $p=0$, let $Q\in \mathcal{K}_0$, $g:\mathbb{S}^n\rightarrow\mathbb{R}$ be a continuous function. Let $\delta>0$ be small enough. For each $t \in (-\delta,\delta)$, define the continuous function $h_t:\mathbb{S}^n\rightarrow(0,\infty)$ as
		\[
		h_t(v)=h_{Q}(v)\cdot e^{tg(v)}, \quad  v \in \mathbb{S}^n.
		\]
		Then
		\[
		\lim_{t\rightarrow 0}\frac{r_{[h_t]}(u)-r_Q(u)}{t}=g(\alpha_Q(u))r_Q(u),
		\]
		holds for almost all $u \in \mathbb{S}^n$. In addition, there exists $M>0$ such that
		\[
		|r_{[h_t]}(u)-r_Q(u)|\leq M|t|,
		\]
		for all $u \in \mathbb{S}^n$ and $t \in (-\delta,\delta)$.
	\end{lemma}
	
	\begin{proof}
		Since
		\[
		h_t(v)=h_{Q}(v)\cdot e^{tg(v)},
		\]
		we have
		\[
		\log h_t(v)=\log h_Q(v)+tg(v).
		\]
		Then the result is an application of  Lemmata 2.8 and 4.1 in \cite{2016HLYZ}.
	\end{proof}
	
	Now we derive the $L_p$-variational formula for $p=0$. The $L_p$-variational formula for $p\neq 0$ has been obtained in \cite{2021Liu}.
	
	\begin{lemma}\label{p2}
		For $p=0$, let $Q\in \mathcal{K}_0$, $g:\mathbb{S}^n\rightarrow\mathbb{R}$ be a continuous function. Let  $\delta>0$ be  small enough. For each $t \in (-\delta,\delta)$, define the continuous function $h_t:\mathbb{S}^n\rightarrow(0,\infty)$ as
		\[
		h_t(v)=h_{Q}(v)\cdot e^{tg(v)}, \quad  v \in \mathbb{S}^n.
		\]
		Then
		\begin{equation}\label{var}
			\lim_{t\rightarrow 0}\frac{\gamma_{n+1}([h_t])-\gamma_{n+1}(Q)}{t}=\int_{\mathbb{S}^n} g\ dS_{p,\gamma_{n+1},Q}.
		\end{equation}
	\end{lemma}
	
	\begin{proof} The proof is a simply modification of the proof for $p\neq 0$ in \cite{2021Liu}. 
		By use of the polar coordinates, we have
		\[
		\gamma_{n+1}([h_t])=\frac{1}{(\sqrt{2\pi})^{n+1}}\int_{\mathbb{S}^n}\int_0^{r_{[h_t]}(\xi)}e^{\frac{-r^2}{2}}r^n\ drd\xi.
		\]
		%\begin{equation}\label{jia2.15}
		%\gamma_{n+1}([h_t])=\frac{1}{(\sqrt{2\pi})^{n+1}}\int_{\mathbb{S}^n}\int_0^{\rho_{[h_t]}(u)}e^{\frac{-r^2}{2}}r^ndrdu.
		%\end{equation}
		Since $Q\in \mathcal{K}_0$ and $g\in C(\mathbb{S}^n)$, for $t$ close to $0$, there exists $M_1>0$ such that $[h_t]\subset M_1B$. Denote 
		$F(s)=\int_0^s e^\frac{-r^2}{2}r^n dr$. By mean value theorem,
		\[
		|F(r_{[h_t]}(u))-F(r_Q(u))|\leq |F'(\theta)||r_{[h_t]}(u)-r_Q(u)|<M|F'(\theta)||t|,
		\]
		where $M$ comes from Lemma \ref{p1}, and $\theta$ is between $r_{[h_t]}(u)$ and $r_Q(u)$. Since $[h_t]\subset M_1B$, we have $\theta \in (0,M_1]$. Therefore, by definition of $F$, we have $|F'(\theta)|$ is bounded from above
		by some constant which depends on $M_1$. Therefore, there exists $M_2 > 0$ such that
		\[
		|F(r_{[h_t]}(u))-F(r_Q(u))|\leq M_2|t|.
		\]
		Using dominated convergence theorem, together with Lemma \ref{p1} we have
		\begin{align*}
			\lim_{t\rightarrow 0}\frac{\gamma_{n+1}([h_t])-\gamma_{n+1}(Q)}{t} & = \frac{1}{(\sqrt{2\pi})^{n+1}}\int_{\mathbb{S}^n}g(\alpha_Q(u))e^{\frac{-r^2_Q(u)}{2}}r^{n+1}_Q(u)du\\
			&=\frac{1}{(\sqrt{2\pi})^{n+1}}\int_{\partial Q}g(\nu_Q(x))e^{\frac{-|x|^2}{2}}(x\cdot\nu_Q(x))d\mathcal{H}^n(x)\\
			&=\int_{\mathbb{S}^n}g\ dS_{p,\gamma_{n+1},Q}.
		\end{align*}
	\end{proof}

	%\begin{lemma}\label{2.1}%Lemma 2.1
	%The functional (\ref{flow03}) is non-increasing along the flow (\ref{flow02}). Namely $\frac{d}{dt}\Phi(M_t)\leq 0$,and the equality holds if and only if $M_t$ satisfies the elliptic equation (\ref{flow01})
	%\end{lemma}
	
	%\begin{proof}
	%It is easy to see from (\ref{flow02}) and (\ref{flow03})
	%\begin{align*}
	%\frac{d}{dt}\Phi(M_t)&= \int_{\mathbb{S}^n}f(x)u^{p-1}u_tdx-\frac{1}{(\sqrt{2\pi})^{n+1}}\int_{\mathbb{S}^n}e^{-\frac{r^2}{2}}r^n(\xi,t)\partial_tr(\xi,t)d\xi\nonumber\\
	% &=\int_{\mathbb{S}^n}f(x)u^{p-1}u_t-\frac{1}{(\sqrt{2\pi})^{n+1}}e^{-{\frac{|\nabla u|^2+u^2}{2}}}\frac{u_t}{K}dx\nonumber\\
	% &=-(\sqrt{2\pi})^{n+1}\int_{\mathbb{S}^n}e^{\frac{r^2}{2}}Ku(fu^{p-1}-\frac{1}{(\sqrt{2\pi})^{n+1}}\frac{1}{K}e^{-{\frac{|\nabla u|^2+u^2}{2}}})^2dx\nonumber\\
	% &\leq 0.
	%\end{align*}
	%Clearly $\frac{d}{dt}\Phi(M_t)=0$ if and only if
	%\[fu^{p-1}=\frac{1}{(\sqrt{2\pi})^{n+1}}\frac{1}{K}e^{-{\frac{|\nabla u|^2+u^2}{2}}}.\]
	%Namely $M_t$ satisfies (\ref{flow01}).
	%\end{proof}

	\section{The variational method to generate solution}\label{sec3}
	
	By employing variational method, we study the $L_p$-Gaussian Minkowski problem for $p=0$. We prove the optimizer of a suitable variational is just the solution of the Minkowski problem of $S_{p,\gamma_{n+1},K}$ for $p=0$. The method comes from \cite{2016HLYZ, 2021Liu}.
	
	Consider the following minimum problem:
	\[
	\min \left\{\phi(Q):Q\in \mathcal{K}_e,\gamma_{n+1}(Q)=\frac{1}{2}\right\},
	\]
	where $\phi:\mathcal{K}_e \rightarrow \mathbb{R}$ is given by
	\[\phi(Q)=\int_{\mathbb{S}^n}\log h_Q(v)d\mu(v).\]
	
	\begin{lemma}\label{p31}
		If $K\in \mathcal{K}_e$ and satisfies
		\[
		\phi(K)=\min \left\{\phi(Q):Q\in \mathcal{K}_e,\gamma_{n+1}(Q)=\frac{1}{2}\right\},
		\]
		then this fact is equivalent to
		\[
		\varphi(h_K)=\min \left\{\varphi(z):z\in C_e^+(\mathbb{S}^{n}),\gamma_{n+1}([z])=\frac{1}{2}\right\},
		\]
		where $\varphi: C_e^+(\mathbb{S}^{n})\rightarrow \mathbb{R}$ is defined by
		\[
		\varphi(z)=\int_{\mathbb{S}^n}\log z(v)d\mu(v).
		\]
	\end{lemma}
	\begin{proof}
		It is easy to get $h_{[z]}\leq z$ and $[h_{[z]}]=[z]$ from the definition of the Wulff shape, thus we have
		\[
		\varphi(z)\geq \varphi(h_{[z]}).
		\]
		Clearly, $K\in \mathcal{K}_e$ and satisfies
		\[
		\phi(K)=\min \left\{\phi(Q):Q\in \mathcal{K}_e,\gamma_{n+1}(Q)=\frac{1}{2}\right\},
		\]
		if and only if
		\[
		\varphi(h_K)=\min \left\{\varphi(z): z\in C_e^+(\mathbb{S}^{n}),\gamma_{n+1}([z])=\frac{1}{2}\right\}.
		\]
	\end{proof}
	
	\begin{lemma}\label{p32}
		Let $\mu$ be a nonzero finite Borel measure on $\mathbb{S}^{n}$ and $p=0$. If $K\in \mathcal{K}_e$ and satisfies
		\[
		\phi(K)=\min \left\{\phi(Q):Q\in \mathcal{K}_e,\gamma_{n+1}(Q)=\frac{1}{2}\right\},
		\]
		then there exists a constant $c >0$ such that $\mu=c S_{p,\gamma_{n+1},K}$.
	\end{lemma}
	\begin{proof}
		By Lemma \ref{p31} we get
		\begin{equation}\label{p33}
			\varphi(h_K)=\min \left\{\varphi(z):z\in C_e^+(\mathbb{S}^{n}),\gamma_{n+1}([z])=\frac{1}{2}\right\}.
		\end{equation}
		For any $g\in C^+_e(\mathbb{S}^n)$ and $t\in (-\delta,\delta)$ where $\delta>0$ is  sufficiently small, let
		\[
		h_t(v)=h_{K}(v)\cdot e^{tg(v)}.
		\]
		So (\ref{p31}) implies that $h_K$ being a minimizer. By the Lagrange multipliers argument, we may conclude that there exists a constant $c > 0$ such that
		\[
		\frac{d}{dt}\Big|_{t=0}\varphi(h_t)=c\frac{d}{dt}\Big|_{t=0}\gamma_{n+1}([h_t]).
		\]
		From the variational formula (\ref{var}), we have
		\[
		\int_{\mathbb{S}^n}g(v)d\mu(v)=c\int_{\mathbb{S}^n}g(v)dS_{0,\gamma_{n+1},K}.
		\]
		By the arbitrariness of $g$, we have $\mu=c S_{0,\gamma_{n+1},K}$.
	\end{proof}
	Before the proof of Theorem \ref{p0}, we also need  Lemma 6.2 of \cite{2013BLYZ}.
	\begin{lemma}\label{lem6.2}
		Suppose $\mu$ is a probability measure on $\mathbb{S}^{n}$ that satisfies the strict subspace
		concentration inequality. For each positive integer $l$, let $u_{1,l},\cdot\cdot\cdot,u_{n+1,l}$ be an orthonormal
		basis of $\mathbb{R}^{n+1}$, and suppose that $h_{1,l},\cdot\cdot\cdot,h_{n+1,l}$ are $n+1$ sequences of positive real numbers such
		that $h_{1,l}\leq\cdot\cdot\cdot\leq h_{n+1,l}$, and such that the product $h_{1,l}\cdot\cdot\cdot h_{n+1,l}\geq1$, and $\lim_{l\rightarrow\infty}h_{n+1,l} = \infty $.
		Then, for the cross-polytopes $Q_l=[\pm h_{1,l}u_{1,l},\cdot\cdot\cdot,\pm h_{n+1,l}u_{n+1,l}]$, the sequence
		\[\Phi_\mu(Q_l)=\int_{\mathbb{S}^{n}}\log h_{Q_l}d\mu,\]
		is not bounded from above.
	\end{lemma}

	\begin{proof}[Proof of Theorem \ref{p0}]
		By Lemma \ref{p32}, it suffices to show that a minimizer of the optimization problem  exists. Without loss of generality, assume that $\mu$ is a probability measure. Take $\{Q_i\}$ is a sequence of $O$-symmetric convex bodies such that $\gamma_{n+1}(Q_i)=\frac{1}{2}$ and
		\begin{equation}\label{p34}
			\lim_{i\rightarrow\infty}\phi(Q_i)=\min \left\{\phi(Q):Q\in \mathcal{K}_e,\gamma_{n+1}(Q)=\frac{1}{2}\right\}.
		\end{equation}
		By John's theorem \cite{1937John}, there exists an ellipsoid $E_i$ centered at the origin such that
		\[E_i\subset Q_i\subset \sqrt{n+1}E_i.\]
		Let $e_{1,i},\cdot\cdot\cdot,e_{n+1,i}\in\mathbb{S}^n$ be the principal directions of $E_i$  indexed to satisfy
		\[h_{1,i}\leq\cdot\cdot\cdot\leq h_{n+1,i},\quad \text{where}\quad h_{j,i}=h_{E_i}(e_{j,i}),\quad\text{for}\quad j=1,...,n+1.\]
		Next we define the cross-polytope
		\[C_i=[\pm h_{1,i}e_{1,i},...,\pm h_{n+1,i}e_{n+1,i}].\]
		Since $C_i\subset E_i\subset \sqrt{n+1}C_i$, we have
		\[C_i\subset Q_i\subset (n+1)C_i.\]
		So $\gamma_{n+1}((n+1)C_i)\geq\gamma_{n+1}(Q_i)$. 
		By the definition of $\gamma_{n+1}$, and $\gamma_{n+1}(Q_i)=\frac{1}{2}$, we have
		\begin{equation}\label{p35}
			\prod_{j=1}^{n+1}h_{j,i}\geq\frac{(\sqrt{2\pi})^{n+1}}{2^{n+2}(n+1)^{n+1}}.
		\end{equation}
		Set $A=\frac{(\sqrt{2\pi})^{n+1}}{2^{n+2}(n+1)^{n+1}}$.\\
		\indent Suppose that the sequence ${Q_i}$ is not  bounded. Then ${C_i}$ is not  bounded, and thus, for a subsequence,
		\begin{equation}\label{p36}
			\lim_{i\rightarrow\infty}h_{n+1,i}=\infty.
		\end{equation}
		In view of (\ref{p35}) and (\ref{p36}), applying Lemma \ref{lem6.2} to $C_i^{'}=A^{\frac{-1}{n+1}}C_i$ yields that
		${\phi(C_i^{'})}$ is not bounded from above. Thus
		\[\lim_{i\rightarrow\infty}\phi(Q_i)=\infty,\]
		but this contradicts to (\ref{p34}). Then we conclude $Q_i$ is uniformly bounded. By Blaschke selection theorem, $Q_i$ has convergent subsequence, still denoted by $Q_i$, converges to a compact convex $O$-symmetric set $K$ of $\mathbb{R}^{n+1}$. By the continuity of Gaussian volume, we get
		\[\lim_{i\rightarrow\infty}\gamma_{n+1}(Q_i)=\gamma_{n+1}(K)=\frac{1}{2}.\]
		Next we prove that the support function of $K$ has uniform lower bound by contradiction. Assume that there exists a sequence of cross-polytopes $K_l=[\pm h_{1,l}u_{1,l},\cdot\cdot\cdot,\pm h_{n+1,l}u_{n+1,l}]$, with $\gamma_{n+1}(K_l)=\frac{1}{2}$ and converging to $K$, where
		$u_{1,l},\cdot\cdot\cdot,u_{n+1,l}$ is an orthonormal
		basis of $\mathbb{R}^{n+1}$,  $0<h_{1,l}\leq\cdots\leq h_{n+1,l}$ and $h_{1,l}\rightarrow 0$ as $l\rightarrow \infty$. 
		Therefore,
		\[
		\gamma_{n+1}(K_l)=\frac{1}{(\sqrt{2\pi})^{n+1}}\Pi_{i=1}^{n+1}\left\{\int_{-h_{i, l}}^{h_{i, l}}e^{\frac{-x_i^2}{2}}dx_i\right\}
		\rightarrow 0, \, \, \text{as}\, \,  l\rightarrow \infty
		\]
		which contradicts with the given condition that $\gamma_{n+1}(K_l)=\frac{1}{2}$.
		We therefore conclude that $K$ is non-degenerate. Namely, $K$ is the desired convex body.
	\end{proof}

	\section{Proof of Theorem \ref{yibanfangchengthm} $\&$ \ref{yibanfangchengthm2}}\label{proofofthm2}
	Firstly, we establish the uniform positive upper and lower bounds for the solutions to the flow \eqref{yibanfangcheng2}.
	\begin{lemma}\label{C0estimateC}
		Let $h(\cdot,t),t\in[0,T)$, be a smooth, uniformly convex solution to (\ref{8}). 
		For $p>0$, then there is a positive constant $C$ depending only $p$, and the lower and upper bounds of f and $h(\cdot,0)$ such that
		\begin{equation}\label{c1}
			1/C\leq h(\cdot,t)\leq C, \quad    \forall t \in [0,T).
		\end{equation}
		and
		\begin{equation}\label{gra1}
			\max_{\mathbb{S}^n}|\nabla h|(\cdot,t)\leq C, \quad    \forall t \in [0,T).
		\end{equation}
	\end{lemma}
	
	\begin{proof}
		
		We first prove the upper bound:
		From Lemma 2.6 in \cite{2021ChenLi}, at any fixed time $t$, we have $h(x, t)\geq (x\cdot x^t_{\max})h_{\max}(t),  \forall x \in \mathbb{S}^n$, where $x^t_{\max}$ is the point such that $h(x^t_{\max})=\max_{\mathbb{S}^n}h(\cdot,t)$ and set $h_{\max}(t)=\max_{\mathbb{S}^n}h(\cdot,t)$.
		
		By Lemma \ref{2.2}, we have for the case $p>0$
		\begin{align*}
			\Psi(M_0)&\geq\Psi(M_t)=\frac{1}{p}\int_{\mathbb{S}^n}f(x)h^p(x,t)dx\nonumber\\
			&\geq\frac{1}{p}\int_{\{x\in\mathbb{S}^n:x\cdot x^t_{\max}\geq1/2\}}f(x)h^p(x,t)dx\nonumber\\
			&\geq\frac{1}{p2^p}\int_{\{x\in\mathbb{S}^n:x\cdot x^t_{\max}\geq1/2\}}f(x)h^p_{\max}(t)dx\nonumber\\
			&\geq Ch^p_{\max}(t).
		\end{align*}
		Next we derive a positive lower bound:
		Suppose for a sequence of $\{t_i\}$ with $t_i\rightarrow T \le \infty$ as $i\rightarrow \infty$, there exists $v_i\in \mathbb{S}^n$ such that the support function of $M_{t_i}$, $h(v_i, t_i)\rightarrow 0$, as $i\rightarrow \infty$. Then for any $\epsilon>0$, there exists a large integer $N>0$, such that when $i>N$,  $\Omega_{t_i}\subset\{x\in \R^{n+1}: x\cdot v_i\le \epsilon\}:=H_{\epsilon}$, and $H_{\epsilon}\rightarrow H_0:=H$ as $\epsilon\rightarrow 0$, the half space of $\R^{n+1}$, where $\Omega_{t_i}$ denotes the convex body enclosed by $M_{t_i}$. Combining the fact that the support function of $M_{t_i}$ has upper bound, which means that there exists a large number $R>0$, $\Omega_{t_i}\subset B_R\cap H_{\epsilon}:=D_{\epsilon}$,  we have
		\begin{align*}
			\frac{1}{2}+\frac{\epsilon}{\sqrt{2\pi}}
			&\ge\frac{1}{(\sqrt{2\pi})^{n+1}}\int_H e^\frac{-|x|^2}{2}dx+
			\frac{1}{\sqrt{2\pi}}\int_0^{\epsilon} e^{-\frac{x_i^2}{2}}dx_i
			\frac{1}{(\sqrt{2\pi})^{n}}\int_{\R^n} e^\frac{-|\hat{x}|^2}{2}d\hat{x}\\
			&=\frac{1}{(\sqrt{2\pi})^{n+1}}\int_{H_{\epsilon}}e^\frac{-|x|^2}{2}dx\\
			&=\frac{1}{(\sqrt{2\pi})^{n+1}}\int_{D_{\epsilon}}e^\frac{-|x|^2}{2}dx+\frac{1}{(\sqrt{2\pi})^{n+1}}\int_{H_{\epsilon}\backslash {D_{\epsilon}}}e^\frac{-|x|^2}{2}dx,
		\end{align*}
		where $\hat{x}=(x_1, \cdots, x_{i-1}, x_{i+1}, \cdots, x_{n+1})\in \R^n$. Let $\epsilon\rightarrow 0$, we have $D_{\epsilon}\rightarrow D_0:=D$, and 
		\[
		\frac{1}{2}\ge \frac{1}{(\sqrt{2\pi})^{n+1}}\int_{D}e^\frac{-|x|^2}{2}dx+\frac{1}{(\sqrt{2\pi})^{n+1}}\int_{H\backslash {D}}e^\frac{-|x|^2}{2}dx,
		\]
		which shows that for some $\delta>0$, we have 
		$\frac{1}{(\sqrt{2\pi})^{n+1}}\int_D e^\frac{-|x|^2}{2}dx+\delta\le \frac{1}{2}$. Therefore  
		by Blaschke Selection Theorem and Lemma \ref{gamma-invariant}, we have  $\gamma_{n+1}(\Omega_{t_i})=\gamma_{n+1}(\Omega_{T})\leq\gamma_{n+1}(D)<\frac{1}{2}$, which contradicts the condition $\gamma_{n+1}(\Omega_{t_i})\geq\frac{1}{2}$.\\
		\indent By (\ref{2}), one infers that 
		\[	\max_{\mathbb{S}^n}|\nabla h|(\cdot,t)\leq2\max h(\cdot,x).\]
		
		\noindent Hence (\ref{gra1}) is a consequence of \eqref{c1}.
	\end{proof}
	
	\indent For $-n-1<p\leq0$, we consider the origin-symmetric hypersurfaces and obtain the following $L^{\infty}$-norm estimate.
	\begin{lemma}\label{C0c2}
		Let $M_t$, $t\in[0,T)$ be an origin-symmetric, uniformly convex solution to the flow (\ref{yibanfangcheng2}), and $h(\cdot,t)$ be its support function. If $-n-1<p\leq0$, $f$ is even function, then there is a positive constant $C$ depending only on $p$, and the lower and upper bounds of $f$ and the initial hypersurface $M_0$ such that
		\begin{equation}\label{c2}
			1/C\leq h(\cdot,t)\leq C, \quad    \forall t \in [0,T),
		\end{equation}
		and
		\begin{equation}\label{gra2}
			\max_{\mathbb{S}^n}|\nabla h|(\cdot,t)\leq C, \quad    \forall t \in [0,T).
		\end{equation}
	\end{lemma}

	\begin{proof}
		We first prove the upper bound: at any fixed time $t$, we have $h(x, t)\geq |x\cdot x^t_{\max}|h_{\max}(t),  \forall x \in \mathbb{S}^n$, where $x^t_{\max}$ is the point such that $h(x^t_{\max})=\max_{\mathbb{S}^n}h(\cdot,t)$ and set $h_{\max}(t)=\max_{\mathbb{S}^n}h(\cdot,t)$ (c.f.\cite{2021ChenLi}).
		
		When $p=0$, by a similar argument, we have
		\begin{align*}
			\Psi(M_0)&\geq\Psi(M_t)=\int_{\mathbb{S}^n}f(x)\log h(x,t)dx\nonumber\\
			&\geq\log h_{\max}(t)\int_{\mathbb{S}^n}f(x)dx+\int_{\mathbb{S}^n}f(x)\log |x\cdot x^t_{\max}|dx\nonumber\\
			&\geq C\log h_{\max}(t)-Cf_{\max}.
		\end{align*}
		where $C$ is a positive constant depending only on $f, p$ and $\Omega_0$.
		
		When  $-n-1<p<0$, since $\frac{d}{dt}\Psi(M_t)\leq 0$, there exisit a contant C such that
		\begin{equation}\label{c3}
			C\leq\int_{\mathbb{S}^n}h^p dx.
		\end{equation}
		Suppose there is a sequence origin-symmetric convex body 
		$\Omega_{t_j}$ satisfying (\ref{c3}), but the diameter of $\Omega_{t_j}, d_j=2h_{\max}(t_j)\rightarrow\infty$ as $t_j\rightarrow T$. Let $E_{t_j}$ be the origin-symmetric John ellipsoid associated with $\Omega_{t_j}$, see \cite{2014Schneider}, $\frac{E_{t_j}}{n+1}\subset\Omega_{t_j}\subset E_{t_j},  \frac{h_{E_j}}{n+1}<h_j<h_{E_j}$. we set $\mathbb{S}^n=\mathbb{S}^j_1\bigcup\mathbb{S}^j_2\bigcup\mathbb{S}^j_3$, where
		\[\mathbb{S}^j_1=\mathbb{S}^n\bigcap\{h_{E_j}<\delta\}, \quad \mathbb{S}^j_2=\mathbb{S}^n\bigcap\{\delta\leq h_{E_j}<\frac{1}{\delta}\},  \quad \mathbb{S}^j_3=\mathbb{S}^n\bigcap\{ h_{E_j}\geq\frac{1}{\delta}\}.\]		
		where $\delta \in (0,\frac{1}{4})$. Then
		\[C\leq\int_{\mathbb{S}^n}h_j^p dx<\int_{\mathbb{S}^n}(\frac{h_{E_j}}{n+1})^p dx,\]
		since $p<0$.
			
		Suppose $h_j$ attains the maximum at $x_0$, where $x_0\in\mathbb{S}^n$, that is, $h_j(x_0)=\max_{\mathbb{S}^n}h_j$. Since $h_j(y)\geq\frac{1}{2}d_j|x_0\cdot y|$ for any $y\in\mathbb{S}^n$, we obtain $|\mathbb{S}^j_1|,|\mathbb{S}^j_2|\rightarrow0$ as $d_j\rightarrow\infty$.
		
		 Since the Gaussian volume unchanged along the flow (\ref{yibanfangcheng2}), then $\gamma_{n+1}(\Omega_{t_j})\geq\frac{1}{2}$. By the defination of $\gamma_{n+1}(K)$, we can get  there exist a positive constant C such that Vol$(\Omega_{t_j})>1/C$. We shall use the Blaschke-Santal\'{o} inequality 
		\[ \text{Vol}(\Omega)\text{Vol}(\Omega^{*})\leq \text{Vol}(B_1)^2,\]
		where $\Omega$ is the $O$-symmetry convex body enclosing the origin, $\Omega^{*}$ is the polar body of $\Omega$. So we can get
		\begin{equation}\label{Bl}
			\text{Vol}(E^{*}_{t_j})\leq\frac{\text{Vol}(B_1)^2}{\text{Vol}(E_{t_j})}\leq\frac{\text{Vol}(B_1)^2}{\text{Vol}(\Omega_{t_j})}\leq C.
		\end{equation}
		
		As $d_j\rightarrow\infty$, for any fixed $\delta$, by the H\"{o}lder inequality and (\ref{Bl}), we have
		\[\int_{\mathbb{S}^j_1}(\frac{h_{E_j}}{n+1})^p dx\leq(\frac{1}{n+1})^p(\int_{\mathbb{S}^n}\frac{1}{h_{E_j}^{n+1}})^{\frac{-p}{n+1}}|\mathbb{S}^j_1|^{\frac{p+n+1}{n+1}}\leq C|\mathbb{S}^j_1|^{\frac{p+n+1}{n+1}}\rightarrow0.\]
		Noting $|\mathbb{S}^j_2|\rightarrow0$ as $d_j\rightarrow\infty$, and 
		\[\int_{\mathbb{S}^j_3}(\frac{h_{E_j}}{n+1})^p dx\leq\int_{\mathbb{S}^j_3}\big(\frac{1}{(n+1)\delta}\big)^p dx =\big(\frac{1}{(n+1)\delta}\big)^p|\mathbb{S}^j_3|\leq C\delta^{-p}.\]		
		Hence, we have
		\[C\leq o(1)+C\delta^{-p},\]		
		for any $\delta\in(0,\frac{1}{4})$. Let $\delta\rightarrow0$, we reach a contradiction. It implies $\max_{\mathbb{S}^n}h(\cdot,t)\leq C$, for some positive constant $C$.\\
		\indent A positive lower bound and (\ref{gra2}) can be obtained by a same argument as Lemma \ref{C0estimateC}.
		
	\end{proof}
	
	\begin{lemma}\label{C3}
		Let $X(\cdot,t), t\in[0,T)$ be a uniformly convex solution to (\ref{yibanfangcheng2}). Let $h$ and $r$ be its support function and radial function. Then
		\begin{equation}\label{C31}
			\min_{\mathbb{S}^n\times[0,T)}h\leq r(\cdot,t)\leq \max_{\mathbb{S}^n\times[0,T)}h, \quad    \forall t \in [0,T),
		\end{equation}
		and
		\begin{equation}\label{C32}
			|\nabla r(\cdot,t)|\leq C, \quad    \forall t \in [0,T),
		\end{equation}
		where $C>0$ depends only on $\min_{\mathbb{S}^n\times[0,T)}h$ and $\max_{\mathbb{S}^n\times[0,T)}h$.
	\end{lemma}
	
	\begin{proof}
		Estimate (\ref{C31}) follow from $\max_{\mathbb{S}^n}h(\cdot,t)=\max_{\mathbb{S}^n}r(\cdot,t)$ and $\min_{\mathbb{S}^n}h(\cdot,t)=\min_{\mathbb{S}^n}r(\cdot,t)$. Estimate (\ref{C32}) follows from (\ref{C31}) and (\ref{11}) that we have $|\nabla r(\cdot,t)|\leq \frac{r^2}{h}$.
	\end{proof}

	From Lemmas \ref{C0estimateC}-\ref{C3}, we can get the following estimate about $\theta(t)$.
	\begin{lemma}\label{C00}
		Let $p>-n-1$, there exist a positive constant $C$	independent of $t$, such that for every $t\in[0,T)$,
		\[1/C\leq\theta(t)\leq C.\]
	\end{lemma}
	\begin{proof}
		By the defination of $\theta(t)$:
		\[	\theta(t)=\int_{\mathbb{S}^n}e^{-\frac{r^2}{2}}r^{n+1}(\xi,t)d\xi\bigg/\int_{\mathbb{S}^n}h^p(x,t)f(x)dx,\]
		now the conclusion of this lemma follows directly form Lemmas \ref{C0estimateC}-\ref{C3}.
	\end{proof}		
	We next derive an upper bound for Gauss curvature $K(\cdot,t)$ of the hypersurface $M_t$ evolved by (\ref{yibanfangcheng2}). By (\ref{7}), this is equivalent to show the lower bound of $\det(h_{ij}+h\delta_{ij})$ for the support function $h(\cdot,t)$ of $M_t$.
	\begin{lemma}\label{C2u}
		Let $X(\cdot,t)$ be a uniformly convex solution to the flow (\ref{yibanfangcheng2}) which encloses the origin for $t \in [0,T)$. Then there is a positive constant $C$ depending only on $p$, $f$, $ \min_{\mathbb{S}^n\times[0,T)}h$ and  $\max_{\mathbb{S}^n\times[0,T)}h$, such that
		\begin{equation}\label{C2u1}
			\det(\nabla^2 h+hI)\geq1/C, \quad    \forall (x,t) \in \mathbb{S}^n\times[0,T).
		\end{equation}
	\end{lemma}		
	\begin{proof}
		This is a consequence of Lemma A.1 in \cite{2021ChenLi}. Note that (\ref{8}) can be written in the form of Lemma A.1 in \cite{2021ChenLi}, provided that $\beta=1,\phi(t)=\theta(t), G(x,z,\bm{p})=f(x)z^pe^{\frac{z^2+|\bm{p}|^2}{2}}$ and $\eta(t)=1.$  Conditions (A.2)-(A.4) follow by Lemmas \ref{C0estimateC}-\ref{C3}. Without loss of generality, we may assume that $\min_{\mathbb{S}^n}\det(\nabla^2 h+hI)(\cdot,t)$ is arbitrarily small. By use of Lemma \ref{C00}, condition (A.5) can be verified. Now the estimation (\ref{C2u1})  follows from Lemma A.1.
	\end{proof}	
	Next lemma shows that the principal radii of curvature of hypersurface $M_t$ evolved by (\ref{yibanfangcheng2}) are bounded from above.		
	\begin{lemma}\label{C2L}
		Let $X(\cdot,t)$ be a uniformly convex solution to the flow (\ref{yibanfangcheng2}) which encloses the origin for $t \in [0,T)$. Then there is a positive constant $C$ depending only on $p$, $f$, $ \min_{\mathbb{S}^n\times[0,T)}h$ and  $\max_{\mathbb{S}^n\times[0,T)}h$, such that
		\begin{equation}\label{C2L1}
			(\nabla^2 h+hI)\leq CI, \quad   \forall (x,t) \in \mathbb{S}^n\times[0,T).
		\end{equation}
	\end{lemma}		
	\begin{proof}
		Instead of estimating the upper bound of the principal radii of $M_t$, we study an expanding flow (\ref{CL2}) by Gauss curvature for the hypersurface of the polar dual  convex body of $\Omega_t$, denoted by $M^*_t$. We only need to prove the lower bound of the principal radii of $M^*_t$.\\
		\indent
		Let us denote by $\Omega_t$ and $\Omega_t^{*}$ the convex polar  dual bodies whose support functions are $h$ and $h^{*}$ respectively, where $\Omega_t^{*}=\{z\in \R^{n+1}: z\cdot y\le 1, \forall y\in \Omega_t\}$. 
		It is well-know (see \cite{2020LSW}) that
		\begin{equation}\label{4.8}
			r(\xi,t)=\frac{1}{h^*(\xi,t)},
		\end{equation}
		hence by (\ref{13}), we obtain the following relation
		\begin{equation}\label{4.9}
			\frac{h^{n+2}(x,t)(h^*(\xi,t))^{n+2}}{K(p)K^*(p^*)}=1,
		\end{equation}
		where $p\in M_t,p^*\in M_t^*$ are the two point satisfying $p\cdot p^*=1$ and $x, \xi$ are respectively the unit outer normal of $M_t$ and $M_t^*$ at $p$ and $p^*$. Therefore by equation (\ref{8}) we obtain the equation for $h^*$,
		\begin{equation}\label{CL2}
			\partial_th^*(\xi,t)=\psi(t,\xi,h^*,\nabla h^*)\det(\nabla^2h^*+h^*I)-h^*,\quad t\in[0,T),
		\end{equation}	
		where 		
		\[\psi(t,\xi,h^*,\nabla h^*)=\theta(t)f\big(\frac{\nabla h^*+h^*\xi}{\sqrt{(h^*)^2+|\nabla h^*|^2}}\big)\frac{(h^*)^{n+3}e^{\frac{1}{2(h^*)^2}}}{(\sqrt{(h^*)^2+|\nabla h^*|^2})^{n+p+1}},\]	
		
		By Lemmas \ref{C0estimateC}-\ref{C00} and by (\ref{4.8}), we see that $h^*(\cdot,t)$ and $ \psi(t,\xi,h^*,\nabla h^*)$ are bounded between two positive constant. Applying Lemma A.2 in \cite{2021ChenLi} to (\ref{CL2}), we conclude that 
		\[(\nabla^2 h^*+h^*I)\geq C^{-1}I, \quad  \forall (x,t) \in \mathbb{S}^n\times[0,T). \]
		
	\end{proof}
	
	As a consequence of the above a priori estimates, one sees that the convexity of the hypersurface $M_t$ is preserved under the flow (\ref{yibanfangcheng2}) and the equation (\ref{8}) is uniformly parabolic. By the $C^0$ estimates and gradient estimates in Lemmata \ref{C0estimateC}, \ref{C0c2} and the $C^2$ estimates in Lemmata \ref{C2u}, \ref{C2L}, we can obtain the H\"{o}lder continuity of $\nabla^2h$ and $h_t$ by Krylov's theory \cite{Krylov}. Estimate for higher derivatives follows from the standard regularity theory of uniformly parabolic equations. Hence we obtain the long time existence and regularity of solution for the flow (\ref{yibanfangcheng2}). The uniqueness of the smooth solution to (\ref{8}) follows by the parabolic comparison principle. We obtain the following theorem:
	\begin{theorem}\label{CP}
		\indent Let $M_0$ be a smooth, closed, uniformly convex hypersurface  in $\mathbb{R}^{n+1}$ enclosing the origin. Let $f$ is a smooth positive function on $\mathbb{S}^n$, then flow (\ref{yibanfangcheng2}) has  a unique smooth, uniformly convex solution $M_t$ for all time, if one of the following is satisfied
		\begin{itemize}
			\item[(i)] $p>0$,
			\item[(ii)] $-n-1<p\leq0$, $M_t$ is origin-symmetric as long as the flow exists.
		\end{itemize}
		Moreover we have the a priori estimates
		\[||h||_{C_{x,t}^{k,m}\big(\mathbb{S}^n\times[0,\infty)\big)}\leq C_{k,m},\]
		where $C_{k,m}>0$ depends only on $k, m, f, p$ and the geometry of $M_0$.
	\end{theorem}

	\begin{proof}[Proof of Theorem \ref{yibanfangchengthm} $\&$ \ref{yibanfangchengthm2}]
		
		By the a priori estimates in Lemmata \ref{C0estimateC} and \ref{C0c2}, there is a constant $C>0$, independent of $t$, such that
		\begin{equation}\label{400}
			|\Psi(X(\cdot,t))|\leq C ,\quad    \forall t \in [0,\infty).
		\end{equation}
		\indent By Lemma \ref{2.2}, we obtain
		\begin{equation}\label{401}
			\lim_{t\rightarrow\infty}\Psi(X(\cdot,t))-\Psi(X(\cdot,0)) =-\int_{0}^{\infty}|\frac{d}{dt}\Psi(X(\cdot,t))| dt
		\end{equation}
		\indent By (\ref{400}), the left hand side of (\ref{401}) is bounded below by $-2C$. Hence there is a sequence $t_j\rightarrow\infty$ such that
		\[\frac{d}{dt}\Psi(X(\cdot,t_j))\rightarrow0 \quad \text{as}\quad t_j\rightarrow\infty\]
		which implies, by Lemma \ref{2.2},  Lemma \ref{C0estimateC} and \ref{C0c2} again, $h(\cdot,t_j)$ converges to a positive and uniformly convex function $h_\infty\in C^\infty(\mathbb{S}^n)$ which satisfies (\ref{yibanfangcheng1}) with c given by
		\[\frac{1}{c}=\lim_{t_j\rightarrow\infty}\theta(t_j)=\int_{\mathbb{S}^n}e^{-\frac{|\nabla h_\infty|^2+h^2_\infty}{2}}h_\infty\det (\nabla^2h_\infty+hI)dx\bigg/\int_{\mathbb{S}^n}h_\infty^p(x)f(x)dx.\]
	\end{proof}

	\section{Proof of Theorem \ref{mainthm} $\&$ \ref{mainthm2}}\label{sec5}
	
	Firstly, we establish the uniform positive upper and lower bounds for the solutions to the flow \eqref{flow02}.
	
	\begin{lemma}\label{C0estimate}
		Let $h(\cdot,t),t\in[0,T)$, be a smooth, uniformly convex solution to (\ref{8-1}). 
		If $p> n+1$ or $p=n+1$ with $f<\frac{1}{(\sqrt{2\pi})^{n+1}}$, then there is a positive constant $C$ depending only $p$, and the lower and upper bounds of f and $h(\cdot,0)$ such that
		\begin{equation}\label{31}
			1/C\leq h(\cdot,t)\leq C, \quad    \forall t \in [0,T).
		\end{equation}
	\end{lemma}
	
	\begin{proof}
		Set $h_{\min}(t)=\min_{x\in \mathbb{S}^n}h(x,t)$. By (\ref{8-1}) we have
		\begin{align*}
			\frac{d}{dt}h_{\min} &\ge  -(\sqrt{2\pi})^{n+1}e^{\frac{h_{\min}^2}{2}}\frac{1}{\det(\nabla^2h_{\min}+h_{\min}I)}h_{\min}^pf+h_{\min} \nonumber\\
			&\geq -(\sqrt{2\pi})^{n+1}e^{\frac{h_{\min}^2}{2}}h_{\min}^{p-n}f+h_{\min}\nonumber\\
			&=-((\sqrt{2\pi})^{n+1}e^{\frac{h_{\min}^2}{2}}h_{\min}^{p-n-1}f-1)h_{\min}.
		\end{align*}
		If $(\sqrt{2\pi})^{n+1}e^{\frac{h_{\min}^2}{2}}h_{\min}^{p-n-1}\leq\frac{1}{\max f}$, we have $\frac{d}{dt}h_{\min}\geq0$. This implies $h(\cdot,t)\geq \min_{\mathbb{S}^n}h(\cdot,0)$.
		If not, we have $(\sqrt{2\pi})^{n+1}e^{\frac{h_{\min}^2}{2}}h_{\min}^{p-n-1}>\frac{1}{\max f}$, then we obtain the uniform lower bound of $h$ provided $p>n+1$ or $p=n+1$ with the assumption $f<1/(\sqrt{2\pi})^{n+1}$.
		So we have
		\[
		h(\cdot,t)\geq \min\left\{ \min_{\mathbb{S}^n}h(\cdot,0),\frac{1}{C^{'}}\right\}.
		\]
		Similarly we have
		\[
		h(\cdot,t)\leq \max\left\{ \max_{\mathbb{S}^n}h(\cdot,0),C^{''}\right\}
		\]
		where $C^{''}$ depends only on $\min_{S^n}f$ and $p$.
	\end{proof}

	\indent For $0<p< n+1$, we consider the origin-symmetric hypersurfaces and give the following $L^{\infty}$-norm estimate.
	\begin{lemma}\label{jiaC0estimate}
		Let $M_t$, $t\in[0,T)$ be an origin-symmetric, uniformly convex solution to the flow (\ref{flow02}), and $h(\cdot,t)$ be its support function. If $0<p< n+1$, $f$ is even function and %(\ref{jia1.10}), (\ref{jia1.11}),%
		(\ref{jia1.12}) hold, then there is a positive constant $C$ depending only $p$, and the lower and upper bounds of $f$ and the initial hypersurface $M_0$ such that
		\begin{equation}\label{jia31}
			1/C\leq h(\cdot,t)\leq C, \quad    \forall t \in [0,T).
		\end{equation}
	\end{lemma}
	
	\begin{proof}
		Let $r_{\min}(t)=\min_{\mathbb{S}^n}r(\cdot,t)$ and $r_{\max}(t)=\max_{\mathbb{S}^n}r(\cdot,t)$. By a rotation of coordinates we may assume that $r_{\max}(t)=r(e_1,t)$. Since $M_t$ is origin-symmetric, the points $\pm r_{\max}(t)e_1\in M_t$. Hence
		\[h(x,t)=\sup\{p\cdot x :p\in M_t\}\geq r_{\max}(t)|x\cdot e_1|, \quad \forall x \in\mathbb{S}^n.\]
		Therefore we know from Lemma \ref{2.1}
		\begin{align}\label{jia32}
			\Phi(M_0) &\geq \Phi(M_t) \nonumber\\
			& \geq\frac{1}{p}(\min_{\mathbb{S}^n}f)r_{\max}^p\int_{\mathbb{S}^n}|x\cdot e_1|^pdx-1\nonumber\\
			& \geq C_1r_{\max}^p-1,
		\end{align}
		this proves the upper bound in (\ref{jia31}).\\
		\indent Next we derive a positive lower bound for $r(\xi,t)$ by contradiction. Assume that $r(\xi,t)$ is
		not uniformly bounded away from 0 which means there exists $t_i$
		\begin{equation}\label{jia33}
			r_{\min}(t_i)\rightarrow 0
		\end{equation}
		%\[r_{\min}(t_i)\rightarrow 0,\]
		as $i\rightarrow\infty,$ where $t_i\in[0,T)$. On the other hand, we have $\Omega_t\subset B_R$, where $\Omega_t$ is denotes the convex body enclosed by $M_t$ and $B_R$ denotes the ball centered ar origin with radius R. We may therefore use Blaschke selection theorem and assume
		(by taking a subsequence) that $\Omega_{t_i}$ converges in Hausdorff metric to a compact convex
		$O$-symmetric set $\Omega$. So by (\ref{jia33}), we know that there exists $\xi_0\in\mathbb{S}^n$ such that $r_\Omega(\xi_0)=r_\Omega(-\xi_0)=0$, which implies $\Omega$ in a lower-dimensional subspace. This means that
		\[r(\xi,t_i)\rightarrow 0,\]
		as $i\rightarrow\infty$ almost everywhere with respect to the spherical Lebesgue measure. 
		By Lemma \ref{2.1}
		\begin{align*}
			&\frac{1}{p}\int_{\mathbb{S}^n}f(x)h^p(x,0)dx-\frac{1}{(\sqrt{2\pi})^{n+1}}\int_{\mathbb{S}^n}d\xi\int_0^{r(\xi,0)}e^{-\frac{s^2}{2}}s^nds\nonumber\\
			&\geq\frac{1}{p}\int_{\mathbb{S}^n}f(x)h^p(x,t)dx-\frac{1}{(\sqrt{2\pi})^{n+1}}\int_{\mathbb{S}^n}d\xi\int_0^{r(\xi,t)}e^{-\frac{s^2}{2}}s^nds\nonumber\\
			&\geq-\frac{1}{(\sqrt{2\pi})^{n+1}}\int_{\mathbb{S}^n}d\xi\int_0^{r(\xi,t)}e^{-\frac{s^2}{2}}s^nds
		\end{align*}
		Combined with bounded convergence theorem, we conclude
		\[\gamma_{n+1}(\Omega_0)-\frac{1}{p}\int_{\mathbb{S}^n}f(x)h^p(x,0)dx\leq\frac{1}{(\sqrt{2\pi})^{n+1}}\int_{\mathbb{S}^n}d\xi\int_0^{r(\xi,t_i)}e^{-\frac{s^2}{2}}s^nds\rightarrow 0\]
		as $i\rightarrow\infty$, which is a contraction to (\ref{jia1.12}).
		
	\end{proof}

	For convex hypersurface, the gradient estimate is a direct consequence of the $L^{\infty}$-norm estimate.
	\begin{lemma}\label{C1estimate}
		Let $h(\cdot,t)$, $t\in[0,T)$, be a smooth, uniformly convex solution to (\ref{8-1}). Then we have the gradient estimate
		\begin{equation}\label{32}
			|\nabla h(\cdot,t)|\leq \max_{\mathbb{S}^n\times[0,T)}h, \quad    \forall t \in [0,T).
		\end{equation}
	\end{lemma}
	
	\begin{proof}
		By the relation (\ref{3}), we have
		\[\max_{\mathbb{S}^n}|\nabla h|^2\leq \max_{\mathbb{S}^n}r^2(\cdot,t)=\max_{\mathbb{S}^n}h^2(\cdot,t).\]
		So
		\[\max_{\mathbb{S}^n}|\nabla h|\leq \max_{\mathbb{S}^n}h(\cdot,t).\]
	\end{proof}
	
	Similarly we have the estimate for the radial function $r$.
	
	\begin{lemma}\label{33}
		Let $X(\cdot,t), t\in[0,T)$ be a uniformly convex solution to (\ref{flow02}). Let $h$ and $r$ be its support function and radial function. Then
		\begin{equation}\label{34}
			\min_{\mathbb{S}^n\times[0,T)}h\leq r(\cdot,t)\leq \max_{\mathbb{S}^n\times[0,T)}h, \quad    \forall t \in [0,T),
		\end{equation}
		and
		\begin{equation}\label{35}
			|\nabla r(\cdot,t)|\leq C, \quad    \forall t \in [0,T),
		\end{equation}
		where $C>0$ depends only on $\min_{\mathbb{S}^n\times[0,T)}h$ and $\max_{\mathbb{S}^n\times[0,T)}h$.
	\end{lemma}
	
	\begin{proof}
		Estimate (\ref{34}) follow from $\max_{\mathbb{S}^n}h(\cdot,t)=\max_{\mathbb{S}^n}r(\cdot,t)$ and $\min_{\mathbb{S}^n}h(\cdot,t)=\min_{\mathbb{S}^n}r(\cdot,t)$. Estimate (\ref{35}) follows from (\ref{34}) and (\ref{11}) that we have $|\nabla r(\cdot,t)|\leq \frac{r^2}{h}$.
	\end{proof}

			When $\theta(t)=(\sqrt{2\pi})^{n+1}$, the proof of the $C^2$ estimate is the same as Lemmas \ref{C2u}, \ref{C2L}. Thus we can get the following lemma directly.
			\begin{lemma}\label{lsbound}
				Let $X(\cdot,t)$ be the  solution of the flow (\ref{flow02}) for $t\in[0,T)$. Then there is a positive constant $C$ depending only on $p$, $f$, $\min_{\mathbb{S}^n\times[0,T)}h$, $\max_{\mathbb{S}^n\times[0,T)}h$, such that the principal curvatures of $X(\cdot,t)$ are bounded from above and below
				\begin{equation}\label{4.7}
					1/C\leq\kappa_i(\cdot,t) \leq C,\quad \forall t\in[0,T)\quad \text{and}\quad i=1,\cdots,n.
				\end{equation}
			\end{lemma}

				As a consequence of the above a priori estimates, one sees that the convexity of the hypersurface $M_t$ is preserved under the flow (\ref{flow02}) and the equation (\ref{8-1}) is uniformly parabolic. By the $L^\infty$-norm estimates and gradient estimates in Lemmata \ref{C0estimate}, \ref{jiaC0estimate} and \ref{C1estimate}, we can obtain the H\"{o}lder continuity of $\nabla^2h$ and $h_t$ by Krylov's theory \cite{Krylov}. Estimate for higher derivatives follows from the standard regularity theory of uniformly parabolic equations. Hence we obtain the long time existence and regularity of solution for the flow (\ref{flow02}). The uniqueness of the smooth solution to (\ref{8-1}) follows by the parabolic comparison principle. We obtain the following theorem:
				\begin{theorem}\label{jiajielun}
					\indent Let $M_0$ be a smooth, closed, uniformly convex hypersurface  in $\mathbb{R}^{n+1}$ enclosing the origin. Let $f$ is a smooth positive function on $\mathbb{S}^n$, then flow (\ref{flow02}) has  a unique smooth, uniformly convex solution $M_t$ for all time, if one of the following is satisfied
					\begin{itemize}
						\item[(i)] $p> n+1$,
						\item[(ii)] $p= n+1$ and $f<\frac{1}{(\sqrt{2\pi})^{n+1}}$,
						\item[(iii)] $0<p<n+1$, $M_t$ is origin-symmetric as long as the flow exists and $f$ is even function satisfying
						%(\ref{jia1.10}), (\ref{jia1.11}),
						(\ref{jia1.12}).
					\end{itemize}
					Moreover we have the a priori estimates
					\[||h||_{C_{x,t}^{k,m}\big(\mathbb{S}^n\times[0,\infty)\big)}\leq C_{k,m},\]
					where $C_{k,m}>0$ depends only on $k, m, f, p$ and the geometry of $M_0$.
				\end{theorem}

				\begin{proof}[Proof of Theorem \ref{mainthm}]
					
					By the a priori estimates in Lemmata \ref{C0estimate} and \ref{jiaC0estimate}, there is a constant $C>0$, independent of $t$, such that
					\begin{equation}\label{5.7}
						|\Phi(X(\cdot,t))|\leq C ,\quad    \forall t \in [0,\infty).
					\end{equation}
					\indent By Lemma \ref{2.1}, for any $T>0$, we obtain
					\begin{align*}
						\Phi(X(\cdot,T))-\Phi(X(\cdot,0)) &=-(\sqrt{2\pi})^{n+1}\int_{0}^{T} \int_{\mathbb{S}^n} e^{\frac{r^2}{2}}Kh(fh^{p-1}-\frac{1}{(\sqrt{2\pi})^{n+1}}\frac{1}{K}e^{\frac{-r^2}{2}})^2dxdt\\
						& \leq -C_0\int_{0}^{T} \int_{\mathbb{S}^n} (fh^{p-1}-\frac{1}{(\sqrt{2\pi})^{n+1}}\frac{1}{K}e^{\frac{-r^2}{2}})^2dxdt.
					\end{align*}
					\indent By (\ref{5.7}), the above inequality implies there exists a subsequence of times $t_j\rightarrow \infty$ such that $M_{t_j}$ converges to a limiting hypersurface which satisfies (\ref{flow01}).
					
					To complete the proof of Theorem \ref{mainthm}, it suffices to show that the solution of (\ref{flow01}) is unique. The equation can be written as
					\begin{equation}\label{5.8}
						fh^{p-1}=\frac{1}{(\sqrt{2\pi})^{n+1}}\det (\nabla^2h+hI)e^{-\frac{{h^2+|\nabla h|^2}}{2}}.
					\end{equation}
					
					Let $h_1$ and $h_2$ be two smooth solutions of (\ref{5.8}). Suppose $F=h_1/h_2$ attains its maximum at $x_0\in\mathbb{S}^n$. Then at $x_0$
					\begin{equation}\label{5.9}
						0=\nabla\log F=\frac{\nabla h_1}{h_1}-\frac{\nabla h_2}{h_2},
					\end{equation}
					and $\nabla^2\log F$ is a negative-semidefinite matrix at $x_0$
					\begin{align}\label{5.11}
						0 & \geq \nabla^2\log F \nonumber\\
						& = \frac{\nabla^2 h_1}{h_1}-\frac{\nabla h_1\otimes \nabla h_1}{h_1^2}-\frac{\nabla^2 h_2}{h_2}+\frac{\nabla h_2\otimes \nabla h_2}{h_2^2} \nonumber\\
						&=\frac{\nabla^2 h_1}{h_1}-\frac{\nabla^2 h_2}{h_2}.
					\end{align}
					By (\ref{5.8}) and (\ref{5.11}) we get at $x_0$
					\begin{align}\label{5.10}
						1 & =\frac{e^{\frac{h^2_1+|\nabla h_1|^2}{2}}\det (\nabla^2 h_2+h_2I)h_2^{1-p}}{e^{\frac{h_2^2+|\nabla h_2|^2}{2}}\det (\nabla^2 h_1+h_1I)h_1^{1-p}} \nonumber\\
						& \geq \frac{e^{\frac{h^2_1+|\nabla h_1|^2}{2}}}{e^{\frac{h_2^2+|\nabla h_2|^2}{2}}}F(x_0)^{p-(n+1)}.
					\end{align}
					If $F(x_0)>1$, we have $h_1(x_0)>h_2(x_0)$, then by (\ref{5.9}), we obtain
					\[\frac{e^{\frac{h^2_1+|\nabla h_1|^2}{2}}}{e^{\frac{h_2^2+|\nabla h_2|^2}{2}}}=e^{\frac{1}{2}\left(\frac{|\nabla h_1|^2}{h_1^2}+1\right)(h_1^2-h_2^2)}>1,\]
					which contradicts to (\ref{5.10}) since $p\ge n+1$. So we obtain $F(x_0)=\max_{\mathbb{S}^n}F\leq1$. Similarly we can show $\min_{\mathbb{S}^n} F\geq1$. Therefore $h_1\equiv h_2$.
				\end{proof}
		In order to prove the convergence of the flow \eqref{flow02}, we require the following better gradient estimate.
			\begin{lemma}\label{ImproGra}
				Let $f=\frac{1}{2(\sqrt{2\pi})^{n+1}}$ and $p\geq n+1$, $X(\cdot,t)$ be a smooth uniformly convex solution to the flow \eqref{flow02}, then there exist positive constant $C$ and $C_0$, depending only on the intial hypersurface and $p$, such that 
				\[\max_{\mathbb{S}^n}\frac{|\nabla r|}{r}\leq Ce^{-C_0t}\]
				for all $t>0$.

			\end{lemma} 
			
			\begin{proof}
				Let $w=\log r$, under a local orthonormal frame, we have 
				\[g_{ij}=e^{2w}(w_iw_j+\delta_{ij})\]
				\[b_{ij}=e^w(1+|\nabla w|^2)^{-\frac{1}{2}}(-w_{ij}+w_iw_j+\delta_{ij})\]
				and
				\[K=\frac{\det b_{ij}}{\det g_{ij}}=(1+|\nabla w|^2)^{-\frac{n+2}{2}}e^{-nw}\det(a_{ij})\]
				where $a_{ij}=-w_{ij}+w_iw_j+\delta_{ij}$.\\
				It is not hard to verify that $w$ satisfies the following PDE
				\begin{equation}\label{e1}
					w_t=-\frac{1}{2}(1+|\nabla w|^2)^{-\frac{p+n+1}{2}}e^{\frac{e^{2w}}{2}}e^{(p-n-1)w}\det(a_{ij})+1
				\end{equation}
				Consider the auxiliary function 
				\[Q=\frac{1}{2}|\nabla w|^2.\]
				At the point where $Q$ attains its spatial maximum, we have
				\[0=\nabla_iQ=\sum w_kw_{ki}\]
				\[0\geq\nabla_{ij}Q=\sum w_{kj}w_{ki}+\sum w_kw_{kij}\]
				Denote $\varrho=\frac{1}{2}(1+|\nabla w|^2)^{-\frac{p+n+1}{2}}e^{\frac{e^{2w}}{2}}e^{(p-n-1)w}$. By differentiating (\ref{e1}), we obtain, at the point achieves its spatial maximum,
				\begin{align*}
					\partial_tQ&=\sum w_kw_{kt}=-\det (a_{ij})\sum w_k\varrho_k-\varrho\sum w_k\nabla_k\det (a_{ij})\nonumber\\
					&=-\varrho\det(a_{ij})(p-n-1+e^{2w})|\nabla w|^2+\varrho\det(a_{ij})\sum a^{ij}\nabla_kw_{ij}w_k
				\end{align*}
				By the Ricci identity, we have
				\[\nabla_kw_{ij}=\nabla_jw_{ik}+\delta_{ik}w_j-\delta_{ij}w_k\]
				Hence
				\begin{align*}
					\partial_tQ&=-\varrho\det(a_{ij})(p-n-1+e^{2w})|\nabla w|^2\nonumber\\
					&+\varrho\det(a_{ij})\sum a^{ij}(Q_{ij}-w_{ik}w_{kj}+w_iw_j-\delta_{ij}|\nabla w|^2)\nonumber\\
					&\leq-\varrho\det(a_{ij})(p-n-1+e^{2w})|\nabla w|^2+\varrho\det(a_{ij})(\max_ia^{ii}-\sum a^{ii})|\nabla w|^2\nonumber\\
					&\leq-C_0Q
				\end{align*}
				We therefore have $Q\leq Ce^{-C_0t},$ wher $C_0$ and $C$ are two positive constants which depends only on $p$ and the geometry of $M_0$.

			\end{proof}

		\begin{proof}[Proof of Corollary \ref{Conver}]			
			By Lemma \ref{ImproGra}, We have that $||\nabla r||\rightarrow 0$ exponentially as $t\rightarrow \infty$. Hence by the interpolation and the a priori estimates, we can get that $r$ converges exponentially to a constant in the $C^\infty$ topology as $t\rightarrow\infty.$

		\end{proof}		
				\begin{proof}[Proof of Theorem \ref{mainthm2}]
					Now we consider the case $0<p<n+1$. Since $f$ is even and $M_0$ is origin-symmetric, the solution remains
					origin-symmetric for $t > 0$. The long time existence of the flow \eqref{flow02} now follows from
					Theorem \ref{jiajielun}. As in the proof of Theorem \ref{mainthm}, $M_t$ converges by a subsequence to the solution of equation (\ref{flow01}).\\%a homothetic limit.

				\end{proof}
				\bigskip
				
				\noindent{\bf{Acknowledgement.} }  The research of the authors has been supported by NSFC (Nos. 12031017 and 11971424).

			\end{document}